\documentclass[12pt]{amsart}

\usepackage{pinlabel}
\usepackage{amsmath,amsfonts,amssymb,latexsym,mathrsfs,stmaryrd}
\usepackage{color}
\usepackage[margin=3cm]{geometry}
\usepackage{graphicx}
\usepackage{graphics}
\usepackage{enumerate}
\usepackage{amscd} 
\usepackage{amsxtra}
\usepackage{epic,eepic}
\usepackage{hyperref} 
\usepackage{enumitem}
\usepackage{MnSymbol}
\usepackage{dsfont}
\usepackage[all]{xy}
\usepackage[alphabetic]{amsrefs}
\usepackage{tikz-cd}

\newtheorem{theorem}{Theorem}[section]

\newtheorem{proposition}[theorem]{Proposition}

\newtheorem{conjecture}[theorem]{Conjecture}

\theoremstyle{definition} 
\newtheorem{defn}[theorem]{Definition}
\newtheorem{definition}[theorem]{Definition}

\newtheorem{remark}[theorem]{Remark}

\newcommand{\on}{\operatorname} 

\newcommand{\bM}{\overline{\mathcal M}}

\title{Degenerations, fibrations and higher rank Landau-Ginzburg models}

\author{Charles F. Doran}
\address{Department of Mathematical and Statistical Sciences, University of Alberta and Center of Mathematical Sciences and Applications, Harvard University}
\email{charles.doran@ualberta.ca}

\author{Jordan Kostiuk}
\address{Department of Mathematics, Brown University}
\email{jordan\_kostiuk@brown.edu}

\author{Fenglong You}
\address{School of Mathmatical Sciences \\ University of Nottingham, \\University Park \\Nottingham, NG7 2RD \\United Kingdom.}
\email{fenglong.you@nottingham.ac.uk}
\address{Department of Mathematics \\ ETH Z\"urich, \\Rämistrasse 101, \\8092 Zürich, \\Switzerland}
\email{fenglong.you@math.ethz.ch}

\begin{document}

\begin{abstract} 
We study semi-stable degenerations of quasi-Fano varieties to unions of two pieces. We conjecture that the higher rank Landau-Ginzburg models mirror to these two pieces can be glued together to lower rank Landau-Ginzburg models which are mirror to the original quasi-Fano varieties. We prove this conjecture by relating their Euler characteristics, generalized functional invariants as well as periods. We also use it to conjecture a relation between the degenerations to the normal cones and the fibrewise compactifications of higher rank Landau-Ginzburg models. Furthermore, we use it to iterate the Doran-Harder-Thompson conjecture and obtain higher codimension Calabi-Yau fibrations. 
\end{abstract}

\maketitle

\tableofcontents

\section{Introduction}

We study a generalization of the conjecture of \cite{DHT} by considering a semi-stable degeneration of a quasi-Fano variety $X$ (instead of a degeneration of a Calabi-Yau in \cite{DHT}) such that the central fibre is a union of two quasi-Fano varieties $X_1\cup _{D_0} X_2$. We would like to study the relation of their mirrors.

Recall that, an (ordinary) Landau-Ginzburg (LG) model $(X^\vee, W)$ is a mirror partner of a quasi-Fano variety $X$ with its smooth anticanonical divisor $D$ such that the generic fibre of $W$ is mirror to $D$. If $D$ is not anticanonical (e.g. $D_0\subset X_i$, for $i=1,2$, in the degeneration of a quasi-Fano variety described above) and we still want to define a mirror ``LG model'' for the pair $(X,D)$ such that the generic fibre is mirror to $D$ (with a smooth anticanonical divisor in $D$), then the generic fibre of the ``LG model'' for the pair $(X,D)$ should also be an LG model. Therefore, one should consider an ``LG model'' whose generic fibre is another LG model. In other words, one should consider a higher rank LG model. 

Therefore, in our set-up, we consider a degeneration of a quasi-Fano variety $X$ which leads to a Tyurin degeneration of a smooth anticanonical divisor $D\subset X$ into $D_1\cup _{D_0\cap D_1=D_0\cap D_2}D_2$ where $D_1\subset X_1$ and $D_2\subset X_2$. We consider the hybrid LG models of $(X_1,D_0+D_1)$ and $(X_2,D_0+D_2)$ studied by S. Lee \cite{Lee24}. In this paper, we will simply refer to hybrid LG models as higher rank LG models. In particular, a rank $1$ LG model is an ordinary LG model, i.e., an LG model in the usual sense. In general, a higher rank LG model is a mirror partner of a quasi-Fano variety with a simple normal crossing anticanonical divisor. When the anticanonical divisor consists of two irreducible components, a higher rank LG model (of rank $2$) consists of a pair $(X^\vee, h)$, where $X^\vee$ is a K\"ahler manifold and $h$ is the multi-potential $h:X^\vee\rightarrow \mathbb C^2$ satisfying certain conditions (see Definition \ref{defn:hybrid-LG} for detail). Then we propose the following gluing conjecture.

\begin{conjecture}\label{conj-gluing-1}
Given a semi-stable degeneration of $(X,D)$ into the union of $(X_1,D_0+D_1)$ and $(X_2,D_0+D_2)$, the rank $1$ LG model of $(X,D)$ can be constructed topologically by gluing two rank $2$ LG models of $(X_1,D_0+D_1)$ and $(X_2,D_0+D_2)$. 
\end{conjecture}

Conjecture \ref{conj-gluing-1} can be considered as a generalization of the conjecture in \cite{DHT} to the log Calabi-Yau setting. We provide several pieces of evidence for this conjecture. In Section \ref{sec:gluing}, we prove the topological version of this conjecture by computing Euler numbers. In other words, if a quasi-Fano variety $X$ is obtained from $X_1\cup_{D_0}X_2$ by smoothing and $(X^\vee,W)$ is obtained by gluing rank $2$ LG models of $(X_1,D_0+D_1)$ and $(X_2,D_0+D_2)$, then

\begin{theorem}[see Theorem \ref{thm:topo-gluing}]
\[
\chi(X)=(-1)^{\mathrm d}\chi(X^\vee,W^{-1}(t)),
\] 
where $\mathrm d=\dim_{\mathbb C}X$ and $t$ is a regular value of $W$.
\end{theorem}
We refer to Theorem \ref{thm:topo-gluing} for the precise statement.

In \cite{DKY}, the Doran-Harder-Thompson conjecture is proved at the level of generalized functional invariants as well as periods for Calabi-Yau complete intersections in toric varieties. In this paper, we generalize it to the quasi-Fano complete intersections in toric varieties. In Section  \ref{sec:example-1} and Section \ref{sec:example-2}, we study two explicit examples in detail. In Section \ref{sec:example-1}, we study a complete intersection of bidegrees $(3,1)$ and $(1,1)$ in $\mathbb P^4\times \mathbb P^1$. In Section \ref{sec:example-2}, we study the quartic threefold in $\mathbb P^4$. In Section \ref{sec:toric}, we study toric complete intersections in general. We show that the generalized functional invariants of $(X^\vee,W)$, $(D^\vee,W_0)$, $(X_1^\vee,h_1)$ and $(X_2^\vee,h_2)$ satisfy a product relation (see Proposition \ref{prop-toric-gen-func-inv-'}). Using this product relation among generalized functional invariants, we obtain the relation among periods.
\begin{theorem}[=Theorem \ref{thm:gluing-period-general}]\label{thm-intro-period}
The following Hadamard product relation holds for relative periods of $(X^\vee,W)$, $(D^\vee,W_0)$, $(X_1^\vee,h_1)$ and $(X_2^\vee,h_2)$:
\[
f_0^{X}(q)\star_q f_0^{D_0}(q)=\frac{1}{2\pi i}\oint f_0^{X_1}(q,y)\star_q f_0^{X_2}(q,y)\frac{dy}{y}.
\]
\end{theorem}
The relation among bases of solutions to Picard-Fuchs equations holds in a similar form.
\begin{theorem}[Theorem \ref{thm:general-case-I-function}]\label{thm-intro-I-function}
The Hadamard product relation among the bases of solutions to Picard-Fuchs equations is
\[
I^{X}(q)\star_q I^{D_0}(q)=\frac{1}{2\pi i}\oint I^{X_1}(q,y)\star_q I^{X_2}(q,y)\frac{dy}{y}.
\]
\end{theorem}

The main difference between Theorem \ref{thm-intro-period}, Theorem \ref{thm-intro-I-function} and \cite{DKY}*{Theorem 1.1 and Theorem 1.2} is the following. In \cite{DKY}, absolute periods are obtained as the gluing (i.e., the residue integral of the Hadamard product) of two relative periods of rank $1$ LG models. In Theorem \ref{thm-intro-period} and Theorem \ref{thm-intro-I-function}, we glue relative periods of two rank $2$ LG models to obtain the relative periods for rank $1$ LG models. By the relative mirror theorem of \cite{TY20b}, relative periods of higher rank LG models are mirror to the genus zero Gromov-Witten invariants of simple normal crossings pairs defined in \cite{TY20c}.

Besides considering an LG model as a gluing of rank $2$ LG models, one may also consider an LG model as a fibrewise compactification of a rank $2$ LG model. More precisely, one can consider a quasi-Fano variety $X$ with a smooth anticanonical divisor $D$ or a simple normal crossing anticanonical divisor $D_1+D_2$. Then the LG model mirrors to $(X,D)$ should be considered as fibrewise compactification of the rank $2$ LG model mirrors to $(X,D_1+D_2)$. In Section \ref{sec:deg-normal-cone}, we consider a degeneration to the normal cone of the quasi-Fano variety $X$ with respect to the divisor $D_1$. In this case, gluing two rank $2$ LG models can be considered as providing a fibrewise compactification of a non-proper LG model of $(X,D_1+D_2)$ to obtain a proper LG model of $(X,D)$. In other words,

\begin{conjecture}[=Conjecture \ref{conj-compactification}]
The ordinary LG model of $(X,D)$ is a (partial) compactification of the rank $2$ LG model of $(X,D_1+D_2)$ and the compactification can be obtained by gluing the rank $2$ LG mirror of $(\mathbb P_{D_1}(N_{D_1}\oplus \mathcal O),D_1+\tilde D_1)$ to the rank $2$ mirror LG model of $(X,D_1+D_2)$.
\end{conjecture}

Furthermore, we use this conjecture to study the iterative Tyurin degenerations of Calabi-Yau varieties. For Calabi-Yau varieties admitting such degenerations, we predict that the mirror has a fibration structure over higher dimensional bases. For Calabi-Yau complete intersections in toric varieties that admit two-step Tyurin degenerations, the mirror will have a Calabi-Yau fibration over $\mathbb P^1\times \mathbb P^1$. 

We illustrate this with an explicit example in Section \ref{sec:iterating-DHT}, where we study an elliptic fibration structure of the mirror quintic threefold over a $\mathbb P^1\times \mathbb P^1$ base.
The fibration is by mirror cubic elliptic curves, mirror to  the cubic curves arising in the pair of two-step Tyurin degenerations: $(5)\to (1)+(4)\to (1)+(1)+(3)$ and $(5)\to (2)+(3)\to (1)+(1)+(3)$. Moreover, the first of these degenerations corresponds to a mirror quartic K3-surface fibration and the second to a mirror sextic K3-surface fibration. There is a perfect correspondence between these K3-fibrations on the mirror quintic threefold compatible with the mirror cubic fibration and the two-step Tyurin degenerations of the quintic itself via the refinement of the nef partition.  More generally, we expect the following.
\begin{conjecture}[=Conjecture \ref{conj-higher-dim-base}]
If a Calabi-Yau variety $X$ (or a log Calabi-Yau variety $X\setminus D$) admits a semi-stable degeneration, connected to a point of maximal unipotent monodromy, such that the cone over the dual intersection complex of the central fibre is of dimension $k$, then the mirror of $X$ (or $X\setminus D$) admits a Calabi-Yau fibration structure with a $k$-dimensional base.
\end{conjecture}

To understand more general degenerations of Calabi-Yau varieties or more general iterative Tyurin degenerations of Calabi-Yau varieties, it is also necessary to consider degenerations of simple normal crossings log Calabi-Yau pairs. Recall that Conjecture \ref{conj-gluing-1} considers a degeneration of a smooth pair. The conjecture can be generalized to the case when we consider a degeneration of a quasi-Fano variety $X$ with a simple normal crossing anticanonical divisor $D$. In this case, the mirror of $(X,D)$ itself is a higher rank LG model. Roughly speaking, the conjecture can be stated as follows. 

\begin{conjecture}\label{conj-gluing-n}
Given a quasi-Fano variety $X$ with a simple normal crossings anticanonical divisor $D$, we consider a semi-stable degeneration of a quasi-Fano variety $X$ into $X_1\cup_{D_0}X_2$. The mirror of $(X,D)$, which is a rank $n$ LG model, can be obtained by topologically gluing two rank $(n+1)$ LG models which are mirror to $X_1$ and $X_2$ with their corresponding simple normal crossing divisors. 
\end{conjecture}

\begin{remark}
Recall that the generic fibre of a rank $n$ LG model is a codimension $n$ Calabi-Yau and the mirrors of Calabi-Yau varieties are Calabi-Yau varieties of the same dimension. Therefore, we may consider the mirrors of Calabi-Yau varieties as rank $0$ LG models, that is, LG models with $0$-dimensional bases which are just Calabi-Yau varieties without superpotentials. From this point of view, we may consider the Doran-Harder-Thompson conjecture as the rank $0$ case of Conjecture \ref{conj-gluing-n}. Conjecture \ref{conj-gluing-1} is a rank $1$ case of Conjecture \ref{conj-gluing-n}.
\end{remark}

To prove a topological gluing formula, we need to have a proper definition of Hodge numbers of higher rank LG models. For simplicity, we consider the case when $n=2$. Its generalization to $n>2$ is straightforward.
Hodge numbers of a rank $2$ LG model $(X^\vee,h=(h_1,h_2))$ are defined in Definition \ref{def-hodge-hybrid} (See also \cite{Lee24}):
\[
h^{p,q}(X^\vee,(h_1,h_2)):=h^{p,q}(X^\vee, h_1^{-1}(t_1)\cup h_2^{-1}(t_2)),
\]
where $t=(t_1,t_2)$ is a regular value of $(h_1,h_2)$.
One may consider this as a generalization of the Hodge numbers of LG models of \cite{KKP}. 
With this definition of Hodge numbers, we prove a topological gluing formula relating Euler numbers (see Theorem \ref{thm:topo-gluing-2}):
\[
\chi(X)=(-1)^{\mathrm d}(\chi(X^\vee,h_1^{-1}(t_1)\cup h_2^{-1}(t_2))).
\]
The relations among their generalized functional invariants and relative periods for toric complete intersections hold just like the rank $1$ case.

\subsection*{Acknowledgements}
C.D. is supported by the Natural Sciences and Engineering Research Council of Canada (NSERC).
F.Y. is supported by  the EPSRC grant EP/R013349/1, Research Council of Norway grant no. 202277, and the European Union’s Horizon 2020 research and innovation programme under the Marie Skłodowska-Curie grant agreement 101025386. We would like to thank Lawrence Barrott, Andrew Harder and Alan Thompson for helpful discussions.

\section{Gluing rank $2$ LG models}\label{sec:gluing}

In \cite{DHT}, Doran-Harder-Thompson conjectured a relation between mirrors of Calabi-Yau varieties and mirrors of quasi-Fano varieties under Tyurin degenerations of Calabi-Yau varieties. In this section, we consider a generalization of \cite{DHT} to degenerations of quasi-Fano varieties.

\subsection{Set-up}
\begin{definition}
A smooth variety $X$ is called quasi-Fano if its anticanonical linear system contains a smooth Calabi-Yau member and $H^i(X,\mathcal O_X)=0$ for all $i>0$.
\end{definition}

\begin{defn}[\cite{DHT}, Definition 2.1]\label{def-LG}
A (proper) Landau-Ginzburg (LG) model of a quasi-Fano variety $X$ is a pair $(X^\vee,W)$ consisting of a K\"ahler manifold $X^\vee$ satisfying $h^1(X^\vee)=0$ and a proper map $W:X^\vee\rightarrow \mathbb C$, where $W$ is called the superpotential.
\end{defn}

The generic fibre of $W$ is expected to be mirror to a smooth anticanonical divisor $D\subset X$. The pair $(X^\vee, W)$ is considered as a mirror to the pair $(X,D)$. It was just called LG model in \cite{DHT}, but we will call it a proper LG model or an ordinary LG model here. A natural question to ask is how to extend this mirror symmetry to the case when $D$ has more than one irreducible component. For this, we need the following generalization of LG models.

\begin{definition}\label{def-hybrid-LG-inductive}
An LG model of rank $1$ is the ordinary LG model in Definition \ref{def-LG}. For $n\geq 2$, an LG model of rank $n$ is a pair $(X^\vee, h)$, where
\[
h:=(h_1,\ldots,h_n): X^\vee\rightarrow \mathbb C^n,
\]
such that 
\begin{itemize}
    \item the generic fibre of $h_i$ together with the restriction of $\hat{h}_i:=(h_1,\ldots,h_{i-1},h_{i+1},\ldots,h_n)$ to this fibre is an LG model of rank $(n-1)$.
    \item by composing with the summation map $\Sigma:\mathbb C^n\rightarrow \mathbb C$, we get a non-proper ordinary LG model
    \[
    W:=\Sigma\circ h: X^\vee\rightarrow \mathbb C.
    \]
\end{itemize}
\end{definition}

\begin{remark}
In this paper, we consider the case when the base of a rank $n$ LG model is $\mathbb C^n$. In general, we can allow the base being more general $n$-dimensional varieties. This will be studied elsewhere.
\end{remark}

Now we assume that $D=D_1+D_2$ is a simple normal crossings anticanonical divisor of $X$ and $D_{12}:=D_1\cap D_2$ is smooth. In \cite{Lee24}, he defined the hybrid LG model which is expected to be mirror to the pair $(X,D_1+D_2)$ where $X$ is a Fano variety. Here, we consider the hybrid LG models for quasi-Fano varieties.

\begin{definition}[\cite{Lee24}, Definition 3.1]\label{defn:hybrid-LG}
A hybrid LG model of a quasi-Fano variety $X$ with a simple normal crossings anticanonical divisor $D=D_1+D_2$ is a pair
\[
(X^\vee,h=(h_1,h_2):X^\vee\rightarrow \mathbb C^2),
\]
where $X^\vee$ is a K\"ahler manifold 
such that
\begin{itemize}
    \item a generic fibre of $h_1$, denoted $D^\vee_1$, with 
    \[
    h|_{D^\vee_1}=h_2: D^\vee_1\rightarrow \mathbb C
    \]
    is mirror to $(D_1,D_{12})$;
   \item    a generic fibre of $h_2$, denoted  $D^\vee_2$, with 
    \[
    h|_{D^\vee_2}=h_1: D^\vee_2\rightarrow \mathbb C
    \]
    is mirror to $(D_2,D_{12})$;
    \item by composing with the summation map $\Sigma:\mathbb C^2\rightarrow \mathbb C$, we get a non-proper ordinary LG model mirror to $(X,D)$
    \[
    W:=\Sigma\circ h: X^\vee\rightarrow \mathbb C.
    \]
\end{itemize}
\end{definition}
We will refer to hybrid LG models in Definition \ref{defn:hybrid-LG} as rank $2$ LG models.

\subsection{Topological gluing}\label{sec:top-gluing}
Given a quasi-Fano variety $X$ with its smooth anticanonical divisor $D$, we consider a semi-stable degeneration of $X$ into $X_1\cup_{D_0}X_2$ such that it also gives a Tyurin degeneration of $D$ into $D_1\cup_{D_{12}} D_2$ where $D_1\subset X_1$, $D_{2}\subset X_2$ and $D_{12}=D_1\cap D_0=D_2\cap D_0=D_1\cap D_2$.

Let $(X_1^\vee,h_1=(h_{11},h_{12}))$ and $(X_2^\vee,h_2=(h_{21},h_{22}))$ be rank $2$ LG models mirror to $(X_{1},D_0+D_1\in|-K_{X_1}|)$ and $(X_2,D_0+D_2\in |-K_{X_2}|)$ respectively. We will glue the rank $2$ LG models along their first factors to form an LG model $(X^\vee, W)$ with a fibration over $\mathbb P^1$:
    \[
\begin{tikzcd}
   X^\vee 
   \arrow[r,->,left,"W"] \arrow[swap]{d}{\pi} & \quad \mathbb C \\
 \mathbb P^1 &
\end{tikzcd}
\]
We choose $r_1$ and $r_2$ such that $|\lambda_1|\leq r_1$ for every $\lambda_1$ in the critical locus of $h_{11}$ and $|\lambda_2|\leq r_2$ for every $\lambda_2$ in the critical locus of $h_{21}$. We assume that the generic fibres of $h_{11}$ and $h_{21}$ are topologically the same LG model which mirrors to $(D_0,D_{12})$. 

Let $t=(t_1,t_2)$ be a regular value of $h_i$. The monodromy symplectmorphism $\phi_{11}$ associated to a clockwise loop around infinity of $t_1$ on $h_1^{-1}(t)$ should be the same as the monodromy sympectmorphism $\phi_{21}^{-1}$ associated to a counter-clockwise loop around infinity of $t_1$ on $h_2^{-1}(t)$. For each $t_2$, then we have a rank $1$ LG model $(h_{i2}^{-1}(t_2),h_{i1})$. Then we describe the gluing process following \cite{DHT}.
we choose local trivialization of $h_{i2}^{-1}(t_2)$, described in \cite{DHT}*{Section 2.2}, over $U_i=\{z\in \mathbb C: |z|>r_i\}$ and let $Q_{i,t_2}=h_{i1}^{-1}(U_i,t_2)$. Following \cite{DHT}, one can glue $h_{12}^{-1}(t_2)$ to $h_{22}^{-1}(t_2)$ along $Q_{1,t_2}$ and $Q_{2,t_2}$ to produce the fibre of $X^\vee$ over $t_2$. As in \cite{DHT}, the gluing produces a $C^\infty$ manifold. Let $t_2\in \mathbb C$ varies, then we obtain $X^\vee$ with the superpotential $W$ such that $W|_{X_1^\vee}=h_{12}$ and $W|_{X_2^\vee}=h_{22}$. The gluing respects the fibrations $h_{11}$ and $h_{21}$, so $X^\vee$ is equipped with a fibration over $\mathbb P^1$.

\begin{remark}
We can consider a rank $2$ LG model as a family of ordinary LG models. Then gluing two rank $2$ LG models can be considered as gluing two families of ordinary LG models to obtain a family of Calabi-Yau varieties and this family of Calabi-Yau varieties is the LG mirror of $(X,D)$. Therefore, the gluing picture here is a family version of the original gluing picture of \cite{DHT}. 
\end{remark}

\begin{remark}
There is an alternative view of the gluing picture. Instead of considering the gluing of two rank $2$ LG models to the mirror of the quasi-Fano variety $X$, we can also consider mirror symmetry in the complement of a smooth anticanonical divisor $D$ in $X$ by forgetting the superpotential $W$. For $X_1\setminus D_1$ (or $X_2\setminus D_2$), the mirror still has a superpotential coming from $D_0$. Then the mirror of the non-compact Calabi-Yau $X\setminus D$ can be obtained by gluing two (ordinary) mirror LG models of two non-compact varieties $X_1\setminus D_1$ and $X_2\setminus D_2$.
\end{remark}

We have the following mirror relation between the Euler numbers of quasi-Fano varieties and the LG model constructed by the gluing picture described above.

\begin{theorem}\label{thm:topo-gluing}
Let $X_1$ and $X_2$ be $\mathrm d$-dimensional quasi-Fano varieties which contain the same quasi-Fano hypersurface $D_0$, such that 
\[
K_{X_1}|_{D_0}=-K_{X_2}|_{D_0} \quad \text{and } K_{D_1}|_{D_{12}}=-K_{D_2}|_{D_{12}},
\]
where $D_{12}=D_0\cap D_1=D_0\cap D_2$.
Let $(X_1^\vee,h_1=(h_{11},h_{12}))$ and $(X_2^\vee,h_2=(h_{21},h_{22}))$ be rank $2$ LG models mirror to $(X_{1},D_0+D_1\in|-K_{X_1}|)$ and $(X_2,D_0+D_2\in |-K_{X_2}|)$ respectively. Suppose that the fibres of $h_1$ and $h_2$ are topologically the same Calabi-Yau manifold, which is topologically mirror to $D_{12}$.
Let $X$ be a quasi-Fano variety obtained from $X_1\cup_{D_0}X_2$ by smoothing and let $D\subset X$ be a Calabi-Yau variety obtained from $D_1\cup_{D_{12}} D_2$ by smoothing. 
Let $(X^\vee,W)$ be the LG model obtained by gluing rank $2$ LG models $(X_1^\vee,h_1)$ and $(X_2^\vee,h_2)$. Then
\[
\chi(X)=(-1)^{\mathrm d}\chi(X^\vee,W^{-1}(t)), \quad \text{and } \chi(D)=(-1)^{\mathrm d-1}\chi(W^{-1}(t))
\]
for $t$ a regular value of $W$, where $\chi$ is the Euler number.
\end{theorem}

\begin{proof}

The equality $\chi(D)=(-1)^{\mathrm d-1}\chi(W^{-1}(t))$ was proved in \cite{DHT}*{Theorem 2.3}. It remains to prove $\chi(X)=(-1)^{\mathrm d}\chi(X^\vee,W^{-1}(t))$.

The gluing picture gives the following relative Mayer-Vietoris sequence
\begin{align}\label{rel-mv-seq}
\cdots \rightarrow H^j(X^\vee,W^{-1}(t);\mathbb C)\rightarrow H^j(X_1^\vee,h_{12}^{-1}(t_1);\mathbb C)\oplus H^j(X_2^\vee,h_{22}^{-1}(t_2);\mathbb C) \\
\notag \rightarrow H^j(X_1^\vee\cap X_2^\vee, h_{12}^{-1}(t_1)\cap h_{22}^{-1}(t_2);\mathbb C)\rightarrow \cdots.
\end{align}
Recall that $X_1^\vee\cap X_2^\vee$ and $h_{12}^{-1}(t_1)\cap h_{22}^{-1}(t_2)$ are fibrations over annuli. Their cohomology can be computed using the Wang sequence \cite{PS}*{Theorem 11.33} and we conclude that $\chi(X_1^\vee\cap X_2^\vee, h_{12}^{-1}(t_1)\cap h_{22}^{-1}(t_2))=0$. Therefore,
we have
\begin{align}\label{chi-X-W}
\chi(X^\vee,W^{-1}(t))=\chi(X_1^\vee,h_{12}^{-1}(t_1))+\chi(X_2^\vee,h_{22}^{-1}(t_2)).
\end{align}
Note that $(X_i^\vee,h_{i2})$ should be considered as the mirror of the complement of the divisor $D_0$ of $X_i$, because we forget about $h_{i1}$ which corresponds to the divisor $D_0$. Therefore, we have
\begin{align}\label{chi-X-i}
\chi(X_i^\vee,h_{i2}^{-1}(t_i))=(-1)^{\mathrm d}\chi(X_i\setminus D_0)=(-1)^{\mathrm d}\chi(X_i)-(-1)^{\mathrm d}\chi(D_0).
\end{align}
Combining (\ref{chi-X-W}) and (\ref{chi-X-i}), we have
\begin{align}\label{chi-X-W-1}
\chi(X^\vee,W^{-1}(t))=(-1)^{\mathrm d}(\chi(X_1)+\chi(X_2)-2\chi(D_0)).
\end{align}
On the other hand, by \cite{Lee06}*{Proposition IV. 6},
\begin{align}\label{chi-X-12}
\chi(X)=\chi(X_1)+\chi(X_2)-2\chi(D_0).
\end{align}
Therefore, (\ref{chi-X-W-1}) and (\ref{chi-X-12}) together imply the identity:
\[
\chi(X)=(-1)^{\mathrm d}\chi(X^\vee,W^{-1}(t))
\]
\end{proof}

\subsection{Periods}
The definition of relative periods for LG models is given in \cite{DKY}. One can define periods for higher rank LG models similarly.

\begin{definition}[\cite{DKY}, Definition 2.2]
Let $f\colon\mathcal{X}\to B$ be a family of Calabi-Yau manifolds. 
The morphism $g\colon B\to\mathcal{M}$, taking a point in the base to the corresponding point in the moduli space $\mathcal{M}$ of the fibre is called the \emph{generalized functional invariant} of the family. 
\end{definition}

\begin{defn}\label{def-period}
Given a (rank $n$) LG model and a choice of holomorphic $d$-form $\omega$, we define the \emph{periods} of the (rank $n$) LG model \emph{relative to the (rank $n$) superpotential and $\omega$} to be the period functions associated to the varying fibres of the (rank $n$) LG model obtained by integrating transcendental cycles across the $d$-form $\omega$.
\end{defn}

We refer to \cite{DKY} for more details about the definition of relative periods and how to compute them via generalized functional invariants. In mirror symmetry, under suitable assumptions, periods are mirror to generating functions of genus zero Gromov-Witten invariants (see, for example, \cite{Givental98} and \cite{LLY}). For LG models which are mirror to smooth pairs $(X,D)$, the relative periods considered in \cite{DKY} are mirror to generating functions of genus zero Gromov-Witten invariants of $(X,D)$ by the relative mirror theorem in \cite{FTY}. Now, if we have higher rank LG models which are mirror to a simple normal crossings pairs $(X,D)$, then the relative periods that we consider in this paper are mirror to genus zero Gromov-Witten invariants of $(X,D)$ (\cite{TY20b}, see also Section \ref{sec:GW}).

\section{Example 1: a complete intersection in $\mathbb P^4\times \mathbb P^1$}\label{sec:example-1}

In this section, we consider an example which is analogous to the example considered in \cite{DKY}*{Section 3}. Let $\tilde {Q}_4$ be a complete intersection of bidegrees (3,1) and (1,1) in $\mathbb P^4\times \mathbb P^1$. The threefold $\tilde {Q}_4$ admits a degeneration
\[
\tilde Q_4\leadsto \tilde X_1\cup_{C} \tilde X_2,
\]
where 
\begin{itemize}
    \item $\tilde X_1$ is a hypersurface of bidegree $(3,1)$ in $\mathbb P^3\times \mathbb P^1$;
    \item $\tilde X_2$ is a complete intersection of bidegrees $(3,0), (1,1)$ in $\mathbb P^4\times \mathbb P^1$;
    \item $C$ is a smooth cubic surface in $\mathbb P^3$.
\end{itemize} Indeed, $\tilde X_1$ is the blow-up of $\mathbb P^3$ along the complete intersection of two cubic surfaces and $\tilde X_2$ is the blow-up of a cubic threefold $C_3$ along the complete intersection of two hyperplanes (degree one hypersurfaces in $C_3$). 

The smooth anticanonical $K3$ surface in $\tilde Q_4$ is a complete intersection of bidegrees $(3,1)$ and $(1,1)$ in $\mathbb P^3\times \mathbb P^1$. The degeneration of $\tilde Q_4$ leads to a Tyurin degeneration of this $K3$ surface into  a hypersurface $D_1$ of bidegree $(3,1)$ in $\mathbb P^2\times \mathbb P^1$ and a complete intersection $D_2$ of bidegrees $(3,0), (1,1)$ in $\mathbb P^3\times \mathbb P^1$. This Tyurin degeneration of the $K3$ surface is the degeneration considered in \cite{DKY}*{Section 3} in one dimensional lower (i.e. in dimension two).

\subsection{Functional invariants}\label{sec:gen-fun-inv-conifold}
The mirrors of $\tilde Q_4$, $\tilde X_1$ and $\tilde X_2$ (relative to their smooth anticanonical divisors) can be written down explicitly following \cite{Givental98}. The mirror of $\tilde Q_4$ is an LG model $(\tilde Q_4^\vee,W)$ where $\tilde Q_4^\vee$ is the compactification of 
\[
\left\{(x_0,x_1,x_2,x_3,y)\in (\mathbb C^*)^5\left|
x_1+x_2+\frac{q_1}{x_0x_1x_2x_3}+\frac{q_0}{y}=1;
x_3+y =1\right.\right\}
\]
and the superpotential 
\begin{align*}
W_1: \tilde Q_4^\vee &\rightarrow \mathbb C\\
 (x_0,x_1,x_2,x_3,y)&\mapsto x_0.    
\end{align*}

The LG model for $\tilde X_1$ is 
\begin{align*}
W_1: \tilde X_1^\vee &\rightarrow \mathbb C\\
 (x_0,x_1,x_2,y_1)&\mapsto y_1+x_0,
\end{align*}
where $\tilde X_1^\vee$ is the fibrewise compactification of
\[
\left\{(x_0,x_1,x_2,y_1)\in(\mathbb C^*)^4\left|x_1+x_2+\frac{q_{1,1}}{x_0x_1x_2}+\frac{q_{0,1}}{y_1}=1\right.\right\}.
\]
\begin{remark}
   By fibrewise compactification, we refer to the Calabi--Yau compactification in the spirit of \cite{Przhi13}, \cite{Przhi17} and \cite{Przhi18}. The computation that we explain here does not depend on the fibrewise compactification, so we do not specifically mention how the fibrewise compactification should be taken. 
\end{remark}
The LG model for $\tilde X_2$ is
\begin{align*}
W_2: \tilde X_2^\vee&\rightarrow \mathbb C\\
 (x_0,x_1,x_2,y_2)&\mapsto  y_2+x_0,
\end{align*}
where $\tilde X_2^\vee$ is the fibrewise compactification of
\[
\left\{(x_0,x_1,x_2,y_2)\in(\mathbb C^*)^4\left|x_1+x_2+\frac{q_{1,2}}{x_0x_1x_2(1-q_{0,2}/y_2)}=1\right.\right\}.
\]

The rank $2$ LG models for $(\tilde X_1,D_1+C)$ and $(\tilde X_2,D_2+C)$ are the following:
\begin{align*}
h_1: \tilde X_1^\vee &\rightarrow \mathbb C^2\\
 (x_0,x_1,x_2,y_1)&\mapsto (x_0,y_1)
\end{align*}
and
\begin{align*}
h_2: \tilde X_2^\vee&\rightarrow \mathbb C^2\\
 (x_0,x_1,x_2,y_2)&\mapsto  (x_0,y_2),
\end{align*}
respectively.

The LG mirror of the cubic surface $C$ (relative to its smooth anticanonical divisor $D_{12}=C\cap D_1=C\cap D_2$) is:
\begin{align*}
W_0: C^\vee&\rightarrow \mathbb C\\
 (x_0,x_1,x_2)&\mapsto  x_0,
\end{align*}
where $C^\vee$ is the fibrewise compactification of
\[
\left\{(x_0,x_1,x_2)\in(\mathbb C^*)^3\left|x_1+x_2+\frac{q_{1,0}}{x_0x_1x_2}=1\right.\right\}.
\]
By performing an appropriate change of variables, we see that all four of these families, which are mirrors to $(\tilde Q_4,K3)$, $(\tilde X_1,D_1+C)$, $(\tilde X_2,D_2+C)$ and $(C,D_{12})$ respectively, are fibred by the mirror cubic curve $E^\vee$ whose defining equation is the following:
\begin{align}\label{def-E-vee}
\left\{(\tilde{x}_1,\tilde{x}_2)\in(\mathbb C^*)^2\left|\tilde{x}_1+\tilde{x}_2+\frac{\tilde\lambda}{\tilde{x}_1\tilde{x}_2}=1\right.\right\}.
\end{align}
This allows us to conclude that they are families of mirror cubic curves and read off their functional invariants. 
For example, for the family $\tilde X_1^\vee$, we make the following change of variables:
\[
\tilde{x}_1=\frac{x_1}{1-q_{0,1}/y_1}, \tilde{x}_2=\frac{x_2}{1-q_{0,1}/y_1}, \tilde \lambda=\frac{q_{1,1}}{x_0(1-q_{0,1}/y_1)^3},
\]
to obtain the mirror cubic family. We then read off the functional invariant for $\tilde X_1^\vee$ as $\lambda_1=\frac{q_{1,1}}{x_0(1-q_{0,1}/y_1)^3}$. Similarly, we compute the functional invariants for $C^\vee$, $\tilde{Q}_4^\vee$ and $\tilde X_2^\vee$ by making appropriate changes of variables to match with the cubic mirror family.
We obtain the following functional invariants for $\tilde{Q}_4^\vee$, $C^\vee$, $\tilde X_1^\vee$ and $\tilde X_2^\vee$ respectively:

\[
\lambda=\frac{q_1}{x_0(1-y)(1-q_0/y)^3}, \quad \lambda_0= \frac{q_{1,0}}{x_0}, \quad \lambda_1=\frac{q_{1,1}}{x_0(1-q_{0,1}/y_1)^3}, \quad \lambda_2=\frac{q_{1,2}}{x_0(1-q_{0,2}/y_2)}.
\]

We glue the second factors of the bases of rank $2$ LG models through the following change of variables:

\begin{align}\label{identification-coni-tran}
q_1=q_{1,0}=q_{1,1}=q_{1,2},\quad q_0=q_{0,1},\quad y=y_1=q_{0,2}/y_2.
\end{align}

\begin{proposition}
Under the identification (\ref{identification-coni-tran}), the following product relation holds among the functional invariants:
\begin{align}\label{equ:gen-fun-inv}
\lambda\cdot \lambda_0=\lambda_1\cdot\lambda_2.
\end{align}
\end{proposition}

The discriminant loci of these fibrations are determined by the functional invariants. 
Note that the mirror cubic family of elliptic curves has singular fibres of types $\textrm{I}_3,\textrm{I}_1,\textrm{IV}^*$ over the points $0,\frac{1}{3^3},\infty$ respectively. 

\textbf{Functional Invariant $\lambda_1$:}
\begin{itemize}
    \item  $\lambda_1^{-1}(0)=\{y^3=0\}\cup\{(q_{0}-y)^3=\infty\}\cup\{x_0=\infty\}$. The first two components then support $\textrm{I}_9$ fibres while the third component supports $\textrm{I}_3$-fibres. 
    \item $\lambda_1^{-1}(\infty)=\{y^3=\infty\}\cup\{x_0=0\}\cup\{(q_{0}-y)^3=0\}$. 
Since $\lambda_1=\infty$ is a type $\textrm{IV}^*$ fibre, it follows that only the locus $x_0=0$ supports fibres with non-trivial monodromy. 
\item $\lambda_1^{-1}(\frac{1}{3^3})$ is described by the locus $x_0(y-q_{0})^3=3^3q_{1}y^3$ and generically supports $\textrm{I}_1$-fibres. 
\end{itemize}

\textbf{Functional Invariant $\lambda_2$:}

\begin{itemize}
    \item $\lambda_2^{-1}(0)=\{x_0=\infty\}\cup\{1-y=\infty\}$; both components then support $\textrm{I}_3$-fibres. 
    \item $\lambda_2^{-1}(\infty)=\{x_0=0\}\cup\{y=1\}$, both of which support $\textrm{IV}^*$-fibres.
    \item $\lambda_2^{-1}(\frac{1}{3^3})$ is described by the locus $x_0(1-y)=3^3q_{1}$ and generically supports $\textrm{I}_1$-fibres. 
\end{itemize}

\textbf{Functional Invariant $\lambda_0$:}
\begin{itemize}
    \item $\lambda_0^{-1}(0)=\{x_0=\infty\}$ is a locus of $\textrm{I}_3$-fibres.
    \item $\lambda_0^{-1}(\infty)=\{x_0=0\}$ is a locus of $\textrm{IV}^*$-fibres.
    \item $\lambda_0^{-1}(\frac{1}{3^3})$ is described by $x_0=3^3q_1$ and generically supports $\textrm{I}_1$-fibres.
\end{itemize}   

\textbf{Functional Invariant $\lambda$:}
\begin{itemize}
    \item $\lambda^{-1}(0)=\{y^3=0\}\cup\{(y-q_0)^3=\infty\}\cup\{x_0=\infty\}\cup\{(1-y)=\infty\}$ with the first two components supporting $\textrm{I}_9$-fibres and the last two components supporting $\textrm{I}_3$-fibres.
    \item The locus $\lambda^{-1}(\infty)=\{x_0=\infty\}\cup\{y=1\}\cup\{(y-q_0)^3\}$  and the first two components support $\textrm{IV}^*$-fibres and the last components supports smooth fibres.
    \item The locus $\lambda^{-1}(\frac{1}{3^3})$ is described by $x_0(1-y)(y-q_0)^3=3^3q_1y^3$ and supports $\textrm{I}_1$-fibres generically. 
\end{itemize} 

Note that the singular loci for the invariant $\lambda$ is essentially the union of those corresponding to $\lambda_1,\lambda_2$ with the $\lambda_0$ factor eliminating excess branching at both $x_0=0$ and $x_0=\infty$. Therefore one may also write the identity (\ref{equ:gen-fun-inv}) as
\[
\lambda=\frac{\lambda_1\lambda_2}{\lambda_0}.
\]

\subsection{Periods}

Recall that, the holomorphic period for the mirror cubic curve family $E^\vee$, defined in (\ref{def-E-vee}), is
\[
f_0^{E^\vee}(\tilde \lambda)=\sum_{d\geq 0} \frac{(3d)!}{(d!)^3}\tilde{\lambda}^d.
\]
Following \cite{DKY}, the relative period for the LG model $(C^\vee,W_0)$ is
\[
f_0^{C^\vee}(q_{1,0},x_0)=f_0^{E^\vee}(\lambda_0)=\sum_{d\geq 0} \frac{(3d)!}{(d!)^3}(q_{1,0}/x_0)^d.
\]
The relative period for the LG model $(\tilde Q_4^\vee, W)$ can be computed as a residue integral of the pullback of
\[
\frac{1}{(1-y)(1-q_0/y)}f_0^{E^\vee}(\tilde\lambda)
\]
by the functional invariant
\[
\lambda=\frac{q_1}{x_0(1-y)(1-q_0/y)^3}.
\]
We have
\begin{align*}
\frac{1}{(1-y)(1-q_0/y)}f_0^{E^\vee}(\lambda)&=\frac{1}{(1-y)(1-q_0/y)}\sum_{d_1\geq 0}\frac{(3d_1)!}{(d_1!)^3}\left(\frac{q_1}{x_0(1-y)(1-q_0/y)^3}\right)^{d_1}\\
&=\frac{1}{(1-q_0/y)}\sum_{d_1,d_{0,1}\geq 0}\frac{(3d_1)!}{(d_1!)^3}\left(\frac{q_1}{x_0(1-q_0/y)^3}\right)^{d_1}\frac{(d_1+d_{0,1})!}{d_1!d_{0,1}!}(y)^{d_{0,1}}\\
&=\sum_{d_1,d_{0,1},d_{0,2}\geq 0}\frac{(3d_1)!}{(d_1!)^3}\left(q_1/x_0\right)^{d_1}\frac{(d_1+d_{0,1})!}{d_1!d_{0,1}!}\frac{(3d_1+d_{0,2})!}{(3d_1)!d_{0,2}!}(y)^{d_{0,1}}(q_0/y)^{d_{0,2}},
\end{align*}
where we used the identity
$$\frac{1}{(1-x)^{k+1}}=\sum_{d=0}^\infty\frac{(d+k)!}{d!k!}x^d.$$
After taking a residue, we obtain the holomorphic relative period of $(\tilde{Q}_4^\vee,W)$: 
\begin{align}\label{period-Q-4}
\notag f_0^{\tilde Q_4^\vee}(q_1,q_0,x_0)&=\frac{1}{2\pi i}\oint \frac{1}{(1-y)(1-q_0/y)}f_0^{E^\vee}(\lambda) \frac{dy}{y}\\
\notag &=\sum_{d_1,d_0\geq 0}\frac{(3d_1)!}{(d_1!)^3}\left(q_1/x_0\right)^{d_1}\frac{(d_1+d_{0})!}{d_1!d_{0}!}\frac{(3d_1+d_{0})!}{(3d_1)!d_{0}!}(q_0)^{d_{0}}\\
&=\sum_{d_1,d_0\geq 0}\frac{(3d_1+d_{0})!(d_1+d_{0})!}{(d_1!)^4(d_{0}!)^2}(q_1/x_0)^{d_1}q_0^{d_{0}}.
\end{align}

\begin{remark}

The relative period (\ref{period-Q-4}) matches with the computation for the relative period of the LG model $(\tilde Q_4^\vee, W)$ when the relative period is pullback from the period of mirror quartic $K3$ surface family. In that case, the computation will be as follows. The smooth anticanonical divisor of $\tilde {Q}_4$ is a $K3$ surface. The period of the mirror $K3$ surface family is the following:
\begin{align}\label{period-tilde-Q-4}
\sum_{d_1,d_0\geq 0}\frac{(3d_1+d_{0})!(d_1+d_{0})!}{(d_1!)^4(d_{0}!)^2}(q_1)^{d_1}q_0^{d_{0}}.
\end{align}
The generalized functional invariant for $(X^\vee, W)$ with respect to the mirror $K3$ surface family is
\[
\lambda=\frac{q_1}{x_0}.
\]
The relative period can be computed by pulling back the period of the mirror $K3$ surface family via the generalized functional invariants. The resulting relative period is precisely (\ref{period-Q-4}). The reason why these two computations match follows from the iterative structure of periods that we explained in \cite{DKY}*{Section 2.2.1}. Since the mirror $K3$ surface family is fibred by the mirror cubic curve family, the period (\ref{period-tilde-Q-4}) of the mirror $K3$  can be computed as the residue of the period of the mirror cubic via the functional invariant. 
\end{remark}

The relative periods for the rank $2$ LG models can also be computed as pullback of the period of the mirror cubic via functional invariants:
\begin{align*}
f_0^{\tilde X_1}(q_{1,1},q_{0,1},y_1,x_0)=&\frac{1}{(1-q_{0,1}/y_1)}f_0^{E^\vee}(\lambda_1)\\
=&\frac{1}{(1-q_{0,1}/y_1)}\sum_{d_1\geq 0}\frac{(3d_1)!}{(d_1!)^3}\left(\frac{q_{1,1}}{x_0(1-q_{0,1}/y_1)^3}\right)^{d_1}\\
=&\sum_{d_1,d_0\geq 0}\frac{(3d_1)!(3d_1+d_0)!}{(d_1!)^3(3d_1)!d_0!}(q_{1,1}/x_0)^{d_1}(q_{0,1}/y_1)^{d_0}\\
=& \sum_{d_1,d_0\geq 0}\frac{(3d_1+d_0)!}{(d_1!)^3 d_0!}(q_{1,1}/x_0)^{d_1}(q_{0,1}/y_1)^{d_0};
\end{align*}

\begin{align*}
f_0^{\tilde X_2}(q_{1,2},q_{0,2},y_2,x_0)=&\frac{1}{(1-q_{0,2}/y_2)}f_0^{E^\vee}(\lambda_2)\\
=&\frac{1}{(1-q_{0,2}/y_2)}\sum_{d_1,d_0\geq 0}\frac{(3d_1)!}{(d_1!)^3}\left(\frac{q_{1,2}}{x_0(1-q_{0,2}/y_2)}\right)^{d_1}\\
=&\sum_{d_1,d_0\geq 0}\frac{(3d_1)!(d_1+d_0)!}{(d_1!)^3(d_1)!d_0!}(q_{1,2}/x_0)^{d_1}(q_{0,2}/y_2)^{d_0}\\
=& \sum_{d_1,d_0\geq 0}\frac{(3d_1)!(d_1+d_0)!}{(d_1!)^4 d_0!}(q_{1,2}/x_0)^{d_1}(q_{0,2}/y_2)^{d_0}.
\end{align*}

Under the identification (\ref{identification-coni-tran}), we can write relative periods as follows:
\[
f_0^{C^\vee}(q_1,x_0)=f_0^{E^\vee}(\lambda_0)=\sum_{d\geq 0} \frac{(3d)!}{(d!)^3}(q_1/x_0)^d;
\]
\[
f_0^{\tilde Q_4^\vee}(q_1,q_0,x_0)=\sum_{d_1,d_0\geq 0}\frac{(3d_1+d_{0})!(d_1+d_{0})!}{(d_1!)^4(d_{0}!)^2}(q_1/x_0)^{d_1}q_0^{d_{0}};
\]
\[
f_0^{\tilde X_1}(q_1,q_0,y,x_0)= \sum_{d_1,d_0\geq 0}\frac{(3d_1+d_0)!}{(d_1!)^3 d_0!}(q_1/x_0)^{d_1}(q_0/y)^{d_0};
\]
\[
f_0^{\tilde X_2}(q_1,y,x_0)=\sum_{d_1,d_0\geq 0}\frac{(3d_1)!(d_1+d_0)!}{(d_1!)^4 d_0!}(q_1/x_0)^{d_1}(y)^{d_0}.
\]
Relative periods satisfy the Hadamard product formula similar to the Calabi-Yau case.
\begin{theorem}\label{thm-period-coni-tran}
The holomorphic relative periods of rank $2$ LG models can be glued together to form the holomorphic relative period of the LG model $(\tilde Q_4^\vee,W)$ with the correction given by the holomorphic relative period of the LG model $(C^\vee,W_0)$. More precisely, the relation is given by the Hadamard product
\[
f_0^{\tilde Q_4}(q_1,q_0,x_0)\star_{q_1} f_0^{C^\vee}(q_1,x_0)=\frac{1}{2\pi i}\oint f_0^{\tilde X_1}(q_1,q_0,y,x_0)\star_{q_1} f_0^{\tilde X_2}(q_1,y,x_0)\frac{dy}{y},
\]
where $\star_{q_1}$ means the Hadamard product with respect to the variable $q_1$ and we used (\ref{identification-coni-tran}) to identify variables.
\end{theorem}

\begin{proof}
The computation is straightforward: 
\begin{align*}
    &\frac{1}{2\pi i}\oint f_0^{\tilde X_1}(q_1,q_0,y,x_0)\star_{q_1} f_0^{\tilde X_2}(q_1,y,x_0)\frac{dy}{y}\\
    =&\frac{1}{2\pi i}\oint \sum_{d_1,d_{0,1},d_{0,2}\geq 0}\frac{(3d_1+d_{0,1})!}{(d_1!)^3 d_{0,1}!} \frac{(3d_1)!(d_1+d_{0,2})!}{(d_1!)^4 d_{0,2}!}(q_1/x_0^2)^{d_1}(q_0/y)^{d_{0,1}}(y)^{d_{0,2}}\frac{dy}{y}\\
    =& \sum_{d_1,d_0\geq 0}\frac{(3d_1+d_0)!}{(d_1!)^3 d_0!} \frac{(3d_1)!(d_1+d_0)!}{(d_1!)^4 d_0!}(q_1/x_0^2)^{d_1}q_0^{d_0}\\
    =& f_0^{\tilde Q_4}(q_1,q_0,x_0)\star_{q_1} f_0^{C^\vee}(q_1,x_0).
\end{align*}
\end{proof}

\begin{remark}\label{rmk-absorb-x-0}
Note that, in $f_0^{\tilde Q_4}$, $f_0^{C^\vee}$, $f_0^{\tilde X_1}$ and $f_0^{\tilde X_2}$, variables $q_1$ and $x_0$ always appear together as the factor $q_1/x_0$. We can also simply absorb the factor $1/x_0$ into $q_1$ by rescaling $q_1$. Then Theorem \ref{thm-period-coni-tran} stays the same.  
\end{remark}

Furthermore, the periods $f_0^{\tilde Q_4}$, $f_0^{C^\vee}$, $f_0^{\tilde X_1}$ and $f_0^{\tilde X_2}$ are solutions to systems of PDEs. For example, $f_0^{\tilde Q_4}$ is a solution to the following system of PDEs:
\begin{align}
   &\left((\theta_{q_1})^4-q_1(3\theta_{q_1}+\theta_{q_0}+1)(3\theta_{q_1}+\theta_{q_0}+2)(3\theta_{q_1}+\theta_{q_0}+3)(\theta_{q_1}+\theta_{q_0}+1)\right)F(q_1,q_0)=0\\
  \notag &\left((\theta_{q_0})^2-q_0(3\theta_{q_1}+\theta_{q_0}+1)(\theta_{q_1}+\theta_{q_0}+1)\right)F(q_1,q_0)=0.
\end{align}
PDEs for $f_0^{C^\vee}$, $f_0^{\tilde X_1}$ and $f_0^{\tilde X_2}$ can be written down in a similar way. We also refer to \cite{DKY} for similar examples. We can write down the bases of solutions to these four systems of PDEs (Note that we absorb the factor $1/x_0$ into $q_1$ as mentioned in Remark \ref{rmk-absorb-x-0}):

\[
    I^{\tilde Q_4}(q_1,q_0)=q_1^H q_0^P\sum_{d_1,d_0\geq 0}^\infty q_1^{d_1}q_0^{d_0}\frac{\prod_{k=1}^{3d_1+d_0}(3H+P+k)\prod_{k=1}^{d_1+d_0}(H+P+k)}{\prod_{k=1}^{d_1}(H+k)^4\prod_{k=1}^{d_0}(P+k)^2},
\]

\[
 I^{C^\vee}(q_1)=q_1^H\sum_{d_1\geq 0}^\infty q_1^{d_1}\frac{\prod_{k=1}^{3d_1}(3H+k)}{\prod_{k=1}^{d_1}(H+k)^3},
\]

\[
    I^{\tilde X_1}(q_{1,1},x)=q_{1,1}^H x^P\sum_{d_1,d_0\geq 0}^\infty q_{1,1}^{d_1}x^{d_0}\frac{\prod_{k=1}^{3d_1+d_0}(3H+P+k)}{\prod_{k=1}^{d_1}(H+k)^3\prod_{k=1}^{d_0}(P+k)},
\]

\[
    I^{\tilde X_2}(q_{1,2},y)=q_{1,2}^H y^P\sum_{d_1,d_0\geq 0}^\infty q_{1,2}^{d_1}y^{d_0}\frac{\prod_{k=1}^{3d_1}(3H+k)\prod_{k=1}^{d_1+d_0}(H+P+k)}{\prod_{k=1}^{d_1}(H+k)^4\prod_{k=1}^{d_0}(P+k)},
\]
where $H$ is the hyperplane class from $ H^2(\mathbb P^4)$ in the ambient space $\mathbb P^4\times \mathbb P^1$; $P$ is the hyperplane class from $ H^2(\mathbb P^1)$ in the ambient space $\mathbb P^4\times \mathbb P^1$; $q_1^H=e^{H\log q_1}$ and similarly for other prefactors of this form. These are the (relative) $I$-functions of these varieties which are mirror to Gromov--Witten invariants. We refer to Section \ref{sec:GW} for more details of the relations.

They again satisfy the Hadamard product formula. The following theorem follows from a straightforward computation similar to the proof of Theorem \ref{thm-period-coni-tran}.
\begin{theorem}
We have the following Hadamard product relation
\begin{align}\label{Hadamard-I-function-conifold}
    I^{\tilde Q_4}(q_1,q_0)\star_{q_1} I^{C^\vee}(q_1)=\frac{1}{2\pi i}\oint I^{\tilde X_1}(q_1,q_0/y)\star_{q_1} I^{\tilde X_2}(q_1,y)\frac{dy}{y}.
\end{align}
In (\ref{Hadamard-I-function-conifold}), we treat $\log q_1$ as a variable that is independent from $q_1$. Alternatively, one may consider $\bar{I}(q_1,q_0):=I(q_1,q_0)/q_1^Hq_0^P$ and write (\ref{Hadamard-I-function-conifold}) in terms of $\bar{I}$.
\end{theorem}

\section{Example 2: quartic threefold}\label{sec:example-2}

In this section, we consider the semi-stable degeneration of a smooth quartic threefold $Q_4$ into a cubic threefold $C_3$ and the blow-up $\on{Bl}_C\mathbb P^3$ of $\mathbb P^3$ along complete intersection center of degrees $3$ and $4$ hypersurfaces, that is,
\[
Q_4\leadsto C_3\cup_{C_2}\on{Bl}_C\mathbb P^3,
\]
where $C$ is the complete intersection of degrees $3$ and $4$ hypersurfaces in $\mathbb P^3$, and $C_2$ is the common intersection which is a cubic surface. Note that the smooth anticanonical $K3$ surface in $Q_4$ also degenerates. It degenerates into a cubic surface $C_2^\prime$ which is in $C_3$ and a blow-up $\on{Bl}\mathbb P^2$ of $\mathbb P^2$ along a complete intersection of degrees $3$ and $4$ curves. The blown-up variety $\on{Bl}\mathbb P^2$ is a hypersurface in $\on{Bl}_C\mathbb P^3$. Therefore, we degenerate $(Q_4,K3)$ into $(C_3,C_2^\prime)$ and $(\on{Bl}_C \mathbb P^3, \on{Bl} \mathbb P^2)$ intersecting along $C_2$. 

We would like to write down the rank $2$ LG models for $(C_3,C_2+C_2)$ and $(\on{Bl}_C \mathbb P^3, \on{Bl} \mathbb P^2+C_2)$ and glue them to the mirror LG model of $(Q_4,K3)$. Since $\on{Bl}_C\mathbb P^3$ can be written as a hypersurface in a toric variety, we can write down its LG model following Givental \cite{Givental98}. For the rest of this section, we write 
\[
X:=Q_4, \quad D_0:=C_2, \quad X_1:=C_3, \quad \text{and } \tilde X_2:=\on{Bl}_C\mathbb P^3.
\]
We also write $D$ for the smooth anticanonical $K3$ surface in $X$. We write
\[
D_1:=C_2^\prime\subset X_1, \quad D_2:=\on{Bl}\mathbb P^2\subset X_2, \quad \text{and } D_{12}:=D_1\cap D_2.
\]
\subsection{Functional invariants}
The mirror LG model of $X$ is defined by 
\begin{align*}
W_1:X^\vee\rightarrow \mathbb C\\
(x_0,x_1,x_2,y)\mapsto x_0,
\end{align*}
where $X^\vee$ is the fibrewise compactification of
\[
\left\{(x_0,x_1,x_2,y)\in (\mathbb C^*)^4\left| x_1+x_2+y+\frac{q_{1,1}}{x_0x_1x_2y}=1\right.\right\}.
\]
The mirror LG model of the cubic surface $D_0$ is defined by
\begin{align*}
W_1:D_0^\vee\rightarrow \mathbb C\\
(x_0,x_1,x_2)\mapsto x_0,
\end{align*}
where $D_0^\vee$ is the fibrewise compactification of
\[
\left\{(x_0,x_1,x_2)\in (\mathbb C^*)^3\left| x_1+x_2+\frac{q_{1,0}}{x_0x_1x_2}=1\right.\right\}.
\]

The mirror LG model of $X_1$ is
\begin{align*}
W_1:X_1^\vee\rightarrow \mathbb C\\
(x_0,x_1,x_2,y_1)\mapsto x_0+y_1,
\end{align*}
where $X_1^\vee$ is the fibrewise compactification of
\[
\left\{(x_0,x_1,x_2,y_1)\in (\mathbb C^*)^4\left| x_1+x_2+\frac{q_{1,1}}{x_0x_1x_2y_1}=1\right.\right\}.
\]

Following \cite{CCGK}*{Section E}, the quasi-Fano variety $\tilde X_2=\on{Bl}_C(\mathbb P^3)$ can be constructed as a hypersurface of degree (3,1) in the toric variety $\mathbb P(\mathcal O_{\mathbb P^3}(-1)\oplus \mathcal O_{\mathbb P^3})$.
The LG model of $\tilde X_2$ is defined as follows
\begin{align*}
W_2:\tilde X_2^\vee\rightarrow \mathbb C\\
(x_0,x_1,x_2,y_2)\rightarrow x_0+y_2,
\end{align*}
where $\tilde X_2^\vee$ is the fibrewise compactification of
\[
\left\{(x_0,x_1,x_2,y_2)\in (\mathbb C^*)^4\left|x_1+x_2+\frac{q_{0,2}}{y_2}+\frac{q_{1,2}y_2}{x_0x_1x_2}=1\right.\right\}.
\]
The rank $2$ LG models for $(X_1,D_1+D_0)$ and $(\tilde X_2,D_2+D_0)$ are
\begin{align*}
h_1=(h_{11},h_{12}):X_1^\vee\rightarrow \mathbb C^2\\
(x_0,x_1,x_2,y_1)\mapsto (x_0,y_1),
\end{align*}
and
\begin{align*}
h_2=(h_{21},h_{22}):\tilde X_2^\vee\rightarrow \mathbb C^2\\
(x_0,x_1,x_2,y_2)\rightarrow (x_0,y_2)
\end{align*}
respectively. Recall that the mirror cubic curve family is given by the compactification of
\[
\left\{(\tilde{x}_1,\tilde{x}_2)\in(\mathbb C^*)^2\left|\tilde{x}_1+\tilde{x}_2+\frac{\tilde\lambda}{\tilde{x}_1\tilde{x}_2}=1\right.\right\}.
\]
Similar to the computation in Section \ref{sec:gen-fun-inv-conifold}, we obtain the functional invariants for $X^\vee$, $D_0^\vee$, $X_1^\vee$ and $\tilde X_2^\vee$ respectively : 
\begin{align}
    \lambda=\frac{q_1}{x_0y(1-y)^3}, \quad \lambda_0=\frac{q_{1,0}}{x_0}, \quad \lambda_1=\frac{q_{1,1}}{x_0y_1}, \quad \lambda_2=\frac{q_{1,2}y_2}{x_0(1-q_{0,2}/y_2)^3}.
\end{align}
We set
\begin{align}\label{identification-Q_4-1}
q_1=q_{1,0}=q_{1,1}=q_{1,2}y_2, \quad y=y_1=q_{0,2}/y_2.
\end{align}
The matching of singular fibres works similarly as in Section \ref{sec:example-1}. 
We have
\begin{proposition}
Under the identification (\ref{identification-Q_4-1}), the following product relation holds among the functional invariants:
\begin{align}\label{equ:gen_fun_inv2}
\lambda\cdot\lambda_0=\lambda_1\cdot\lambda_2.
\end{align}
\end{proposition}
\subsection{Periods}

Similar to Section \ref{sec:example-1}, relative periods can be computed as pullbacks of the period for the mirror cubic curve family via the functional invariants. The relative period for $(X^\vee, W)$ is
\[
f_0^{X^\vee}(q_1,x_0)=\sum_{d_1\geq 0}(q_1/x_0)^{d_1}\frac{(4d_1)!}{(d_1!)^4}.
\]
The relative period for $(D_0^\vee, W_0)$ is
\[
f_0^{D_0^\vee}(q_{1,0},x_0)=\sum_{d_1\geq 0}(q_{1,0}/x_0)^{d_1}\frac{(3d_1)!}{(d_1!)^3}.
\]
The relative period for $(X_1^\vee,h_1)$ is
\[
f_0^{X_1^\vee}(q_{1,1},y_1,x_0)=f_0^{E^\vee}(\lambda_1)=\sum_{d_1\geq 0}\frac{(3d_1)!}{(d_1!)^3}\left(\frac{q_{1,1}}{x_0y_1}\right)^{d_1}.
\]
The relative period for $(\tilde X_2^\vee,h_2)$ is
\begin{align*}
f_0^{\tilde X_2^\vee}(q_{1,2}y_2,q_{0,2}/y_2,x_0)=&\frac{1}{1-q_{0,2}/y_2}f_0^{E^\vee}(\lambda_2)\\
=&\frac{1}{1-q_{0,2}/y_2}\sum_{d_1\geq 0}\frac{(3d_1)!}{(d_1!)^3}\left(\frac{q_{1,2}y_2}{x_0(1-q_{0,2}/y_2)^3}\right)^{d_1}\\
=& \sum_{d_1,d_0\geq 0}\frac{(3d_1)!}{(d_1!)^3}(q_{1,2}y_2/x_0)^{d_1}\frac{(3d_1+d_0)!}{(3d_1)!d_0!}(q_{0,2}/y_2)^{d_0}\\
=& \sum_{d_1,d_0\geq 0}\frac{(3d_1+d_0)!}{(d_1)!^3d_0!}(q_{1,2}y_2/x_0)^{d_1}(q_{0,2}/y_2)^{d_0}.
\end{align*}

Note that $q_1=q_{1,1}=q_{1,2}y_2$ and $y=y_1=q_{0,2}/y_2$ when we glue the rank $2$ LG models.
We may rewrite the relative period for $(\tilde X_2^\vee,h_2)$ as
\[
f_0^{\tilde X_2^\vee}(q_1,y,x_0)=\sum_{d_1,d_0\geq 0}\frac{(3d_1+d_0)!}{(d_1)!^3d_0!}(q_1/x_0)^{d_1}(y)^{d_0}.
\]

\begin{theorem}
We have the following Hadamard product relation:
\[
f_0^{X^\vee}(q_1,x_0)\star_{q_1} f_0^{D_0^\vee}(q_1,x_0)=\frac{1}{2\pi i}\oint f_0^{X_1^\vee}(q_1,y,x_0)\star_{q_1} f_0^{\tilde X_2^\vee}(q_1,y,x_0)\frac{dy}{y}.
\]
\end{theorem}

As mentioned in Remark \ref{rmk-absorb-x-0}, we can absorb the factor $1/x_0$ into the factor $q_1$. These periods are solutions to systems of PDEs. The bases of solutions for the PDEs are as follows:

\[
I^{Q_4}(q_1)=q_1^{H}\sum_{d\geq 0}\left(\frac{\prod_{k=1}^{4d} (4H+k)}{\prod_{k=1}^d(H+k)^{4}}\right)q_1^d;
\]
\[
 I^{C^\vee}(q_1)=q_1^H\sum_{d_1\geq 0}^\infty \frac{\prod_{k=1}^{3d_1}(3H+k)}{\prod_{k=1}^{d_1}(H+k)^3}q_1^{d_1};
\]
\begin{align*}
    I^{X_1}(q_1,y)=(q_1/y)^H\sum_{d_1\geq 0}^\infty \frac{\prod_{k=1}^{3d_1}(3H+k)}{\prod_{k=1}^{d_1}(H+k)^3}(q_1/y)^{d_1};
\end{align*}
\[
    I^{\tilde X_2}(q_1,y)=q_1^H y^P\sum_{d_1,d_0\geq 0}^\infty \frac{\prod_{k=1}^{3d_1+d_0}(3H+P+k)}{\prod_{k=1}^{d_1}(H+k)^3\prod_{k=1}^{d_0}(P+k)}q_1^{d_1}y^{d_0}.
\]
\begin{theorem}
The following Hadamard product relation holds:
\begin{align}\label{Hadamard-I-function-Q-4}
    I^{Q_4}(q_1)\star_{q_1} I^{C^\vee}(q_1)=\frac{1}{2\pi i}\oint I^{X_1}(q_1,y)\star_{q_1} I^{\tilde X_2}(q_1,y)\frac{dy}{y},
\end{align}
where we set $P=H$. In (\ref{Hadamard-I-function-Q-4}), we treat $\log q_1$ as a variable that is independent from $q_1$. Alternatively, we can consider $\bar{I}(q_1):=I(q_1)/q_1^H$ and write (\ref{Hadamard-I-function-Q-4}) in terms of $\bar{I}$.
\end{theorem}

\section{Toric complete intersections}\label{sec:toric}

\subsection{Set-up}

Let $X$ be a complete intersection in a toric variety $Y$ defined by a generic section of $E=L_0\oplus L_1\oplus\cdots \oplus  L_s$, where each $L_l$ is a nef line bundle for $0\leq l \leq s$. Let $\rho_l=c_1(L_l)$, we assume that $-K_Y-\sum_{l=0}^s \rho_l$ is a sum of toric prime divisors (not necessarily nef). Without this assumption, the Hori-Vafa mirror may not be of the correct dimension. 

Let $\mathbf s_l\in H^0(Y,L_l)$, for $0\leq l \leq s$, be generic sections determining $X$. We consider a refinement of the partition with respect to $L_0$ such that $L_{0,1}, L_{0,2}$  are two nef line bundles and $L_0=L_{0,1}\otimes L_{0,2}$. Let $\rho_{0,1}=c_1(L_{0,1})$ and $\rho_{0,2}=c_1(L_{0,2})$. Then we have two varieties $X_1$ and $X_2$ defined by generic sections of $E_1=L_{0,1}\oplus L_1\oplus\cdots \oplus  L_s$ and $E_2=L_{0,2}\oplus L_1\oplus\cdots \oplus  L_s$ respectively. 

Let $\mathbf s_{0,1}\in H^0(Y,L_{0,1})$ and $\mathbf s_{0,2}\in H^0(Y,L_{0,2})$ be the generic sections determining $X_1$ and $X_2$. 
We can construct a pencil of complete intersections as follows. Let
\[
X^\prime=\cap_{l=1}^s\{\mathbf s_l=0\}.
\]
We assume that $X^\prime$ is connected  and quasi-smooth. Then we consider the pencil
\[
\mathfrak X:=\{t\mathbf s_0-\mathbf s_{0,1}\mathbf s_{0,2}=0\}\cap X^\prime
\]
in $\mathbb A^1\times Y$ where $t$ is a parameter on $\mathbb A^1$. $\mathfrak X$ has singularities along $t=\mathbf s_0=\mathbf s_{0,1}=\mathbf s_{0,2}=0$. We resolve the singularities of $\mathfrak X$ by blowing up $t=\mathbf s_{0,2}=0$ inside $\mathbb A^1\times Y$ and take the proper transform of $\mathfrak X$. This gives a degeneration of $X$ such that the central fibre of the degeneration is the union of $X_1$ and $\tilde X_2$, where $\tilde X_2$ is given by blowing up $X_2$ along $X^\prime\cap \{\mathbf s_0=\mathbf s_{0,1}=\mathbf s_{0,2}=0\}$.
 Two varieties $X_1$ and $\tilde X_2$ intersect along $D_0$ which is a complete intersection in the toric variety $Y$ defined by a generic section of $L_{0,1}\oplus L_{0,2}\oplus L_1\oplus\cdots \oplus  L_s$. 
 
 Let $X$ be quasi-Fano, then its anticanonical linear system contains a smooth Calabi-Yau member. Let $D$ be a smooth anticanonical divisor of $X$. Then the degeneration of $X$ leads to a degeneration of $D$ into $D_1\cup_{D_{12}}\tilde D_2$ where $D_1\subset X_1$, $\tilde D_2\subset \tilde X_2$ and $D_{12}=D_1\cap \tilde D_2\subset D_0$. We do not assume $-K_Y-\sum_{l=0}^s \rho_l$ is nef on $Y$. But we assume that the restriction of $-K_Y-\sum_{l=0}^s \rho_l$ is nef on a toric subvariety $Y^\prime \subset Y$ and $D$ is defined by a generic section of $E^\prime=L_0^\prime\oplus L_1^\prime\oplus\cdots \oplus  L_s^\prime\oplus L_{s+1}^\prime$, where $L_i^\prime$, $0\leq i\leq s+1$, are nef line bundles of $Y^\prime$ and $L_i^\prime$ are restrictions of $L_i$ on $Y^\prime$ for $0\leq i \leq s$. 
 For example, let $X$ be the blow-up of $\mathbb P^3$ along a complete intersection of degree $4$ and $5$ hypersurfaces. It can be realized as a degree (4,1) hypersurface in the toric variety $Y:=\mathbb P(\mathcal O_{\mathbb P^3}(-1)\oplus \mathcal O_{\mathbb P^3})$, see \cite{DKY}*{Section 4}. Note that $-K_X$ is not nef. However, the smooth anticanonical divisor of $X$ is a quartic $K3$ surface in $\mathbb P^3$. Therefore, this example is within our assumption. The Tyurin degeneration of $D$ is obtained by a refinement of the nef partition which is induced by the refinement of the nef partition of $E$. In other words, we consider a refinement of the nef partition with respect to $L_0^\prime$ such that $L_0^\prime=L_{0,1}^\prime\otimes L_{0,2}^\prime$, where $L_{0,1}^\prime, L_{0,2}^\prime$  are two nef line bundles given by the restriction of $L_{0,1}$ and  $L_{0,2}$ to $Y^\prime$ respectively.

We can realize $\tilde{X}_2$ as a complete intersection in a toric variety following \cite{CCGK}*{Section E}. Indeed, it is a hypersurface in the total space of $\pi:\mathbb P_{X_2}(\mathcal O\oplus i^*L_{0,2}^{-1})\rightarrow X_2$ defined by a generic section of the line bundle $\pi^*i^*L_{0,1}\otimes \mathcal O(1)$, where $i:X_2\hookrightarrow Y$ is the inclusion map. In other words, it is a complete intersection in the toric variety $\mathbb P_Y(\mathcal O\oplus L_{0,2}^{-1})$ given by a generic section of $\pi^*L_{0,2}\oplus \pi^*L_1\oplus\cdots \oplus  \pi^*L_s \oplus (\pi^*L_{0,1}\otimes \mathcal O(1))$ where we use the same $\pi$ for the projection of $\mathbb P_Y(\mathcal O\oplus L_{0,2}^{-1})$ to the base $Y$. 

Let $p_1,\ldots,,p_r\in H^2(Y,\mathbb Z)$ be a nef integral basis. We write the toric divisors as
\[
\mathbf D_j=\sum_{i=1}^r m_{ij}p_i,\quad 1\leq j\leq m, 
\]
for some $m_{ij}$.

The partition of the toric divisors gives a partition of the variables $x_1,\ldots, x_m$ into $s+2$ groups. Let $F_l(x)$ be the sum of $x_i$ in each group for $l=0,\ldots,s,s+1$. Let $F_{0,1}(x)$ and $F_{0,2}(x)$ correspond to the refinement of the nef partition with respect to $L_0$. We write $F_{0,1}(x)=x_{j_1}+\ldots,x_{j_a}$ and $F_{0,2}(x)=x_{j_{a+1}}+\ldots,x_{j_{a+b}}$. Note that we have $F_{0,1}(x)+F_{0,2}(x)=F_0(x)$.

Following \cite{Givental98}, the Hori-Vafa mirror of $X$ is the LG model
\begin{align*}
W:X^\vee &\rightarrow \mathbb C\\
(x_1,\ldots,x_m)&\mapsto F_{s+1}(x),
\end{align*}
where $X^\vee$ is the fibrewise Calabi-Yau compactification of 
\[
\left\{(x_1,\ldots,x_m)\in (\mathbb C^*)^m\left| \prod_{j=1}^m x_j^{m_{ij}}=q_i, i=1,\ldots,r; F_l(x)=1,l=0,\ldots,s\right.\right\}.
\]

The mirror for $D_0$ is the LG model 
\begin{align*}
W_0:D^\vee_0 &\rightarrow \mathbb C\\
(x_1,\ldots,x_m)&\mapsto F_{s+1}(x),
\end{align*}
where $D_{0}^\vee$ is the fibrewise Calabi-Yau compactification of 
\begin{align*}
\left\{(x_1,\ldots,x_m)\in (\mathbb C^*)^m\left| \prod_{j=1}^m x_j^{m_{ij}}=q_i, i=1,\ldots,r;\right.\right. \\
\left.\vphantom{\prod_{j=1}^m} F_l(x)=1,l=1,\ldots,s; F_{0,1}(x)=1, F_{0,2}(x)=1\right\}.
\end{align*}

The mirror for $(X_1,D_0+D_1)$ is the rank $2$ LG model
\begin{align*}
h_1:X_1^\vee &\rightarrow \mathbb C^2\\
(x_1,\ldots,x_m)&\mapsto (F_{s+1}(x),F_{0,2}(x)),
\end{align*}
where $X_1^\vee$ is the fibrewise compactification of
\begin{align*}
&\left\{(x_1,\ldots,x_m)\in (\mathbb C^*)^{m}\left| \prod_{j=1}^m x_j^{m_{ij}}=q_{i,1}, i=1,\ldots,r;\right.\right.\\ &\left.\vphantom{\prod_{j=1}^m} F_l(x)=1,l=1,\ldots,s;F_{0,1}(x)=1\right\}.
\end{align*}

$\tilde{X}_2$ is a complete intersection in the toric variety $\mathbb P_Y(\mathcal O\oplus L_{0,2}^{-1})$ given by a generic section of $L_{0,2}\oplus L_1\oplus\cdots \oplus  L_s \oplus (\pi^*L_{0,1}\otimes \mathcal O(1))$. The rank $2$ LG model for $(\tilde{X}_2,D_0+\tilde{D}_2)$ is
\begin{align*}
h_2:\tilde{X}_2^\vee &\rightarrow \mathbb C^2\\
(x_1,\ldots,x_m,y_2)&\mapsto (F_{s+1}(x),y_2),
\end{align*}
where $\tilde{X}_2^\vee$ is the fibrewise compactification of
\begin{align*}
&\left\{(x_1,\ldots,x_m,y_2)\in (\mathbb C^*)^{m+1}\left| \left(\prod_{j=1}^m x_j^{m_{ij}}\right)y_2^{-\sum_{k=a+1}^{a+b}m_{ij_k}}=q_{i,2}, i=1,\ldots,r;\right.\right.\\ &\left.\vphantom{\prod_{j=1}^m} F_l(x)=1,l=1,\ldots,s;F_{0,1}(x)+\frac{q_{0,2}}{y_2}=1,F_{0,2}(x)=1 \right\}.
\end{align*}

\subsection{Generalized functional invariants}
Then, we can compute the generalized functional invariants. Recall that $X^\vee$ is the compactification of
\[
\left\{(x_1,\ldots,x_m)\in (\mathbb C^*)^m\left| \prod_{j=1}^m x_j^{m_{ij}}=q_i, i=1,\ldots,r; F_l(x)=1,l=0,\ldots,s\right.\right\}.
\]
Set $F_{0,2}(x)=y$, then $F_{0,1}(x)=1-y$. For $y\neq 1$, the following change of variables can give us $D_0^\vee$.
\[
\bar{x}_{j_k}=\frac{x_{j_k}}{1-y}, k=1,\ldots,a;\quad \bar{x}_{j_k}=\frac{x_{j_k}}{y}, k=a+1,\ldots,a+b; 
\]
and 
\begin{align*}
\lambda_i=\frac{q_i}{\prod_{k=1}^a(1-y)^{m_{ij_k}}\prod_{k=a+1}^{a+b}(y)^{m_{ij_k}}}, i=1\ldots, r.
\end{align*}
Therefore, the generalized functional invariant for $X^\vee$ is $\vec \lambda=(\lambda_1,\ldots,\lambda_r)$, where
\begin{align}\label{gen-fun-inv-X}
\lambda_i=\frac{q_i}{\prod_{k=1}^a(1-y)^{m_{ij_k}}\prod_{k=a+1}^{a+b}(y)^{m_{ij_k}}}, i=1\ldots, r.
\end{align}
For $X_1^\vee$, we also set $F_{0,2}(x)=y_1$. By similar computation, the generalized functional invariant is $\vec \lambda_1=(\lambda_{1,1},\ldots,\lambda_{r,1})$, where
\begin{align}\label{gen-fun-inv-X-1}
\lambda_{i,1}=\frac{q_{i,1}}{\prod_{k=a+1}^{a+b}y_1^{m_{ij_k}}}, i=1\ldots, r.
\end{align}
Similarly, the generalized functional invariant for $\tilde{X}_2$ is $\vec \lambda_2=(\lambda_{1,2},\ldots,\lambda_{r,2})$, where
\begin{align}\label{gen-fun-inv-X-2}
\lambda_{i,2}=\frac{q_{i,2}}{y_2^{-\sum_{k=a+1}^{a+b}m_{ij_k}}\prod_{k=1}^{a}(1-q_{0,2}/y_2)^{m_{ij_k}}}, i=1\ldots, r.
\end{align}

We set
\begin{align}\label{identification-X}
q_i=q_{i,1}=q_{i,2}y_2^{\sum_{k=a+1}^{a+b}m_{ij_k}}, i=1,\ldots,r, \quad y=y_1=q_{0,2}/y_2.
\end{align}
\begin{proposition}\label{prop-toric-gen-func-inv}
Under the identification (\ref{identification-X}), the generalized functional invariants satisfy the product relation
\[
\lambda_i q_i=\lambda_{i,1}\lambda_{i,2}, \quad i=1,\ldots, r.
\]
\end{proposition}

\begin{remark}
Note that the generalized functional invariants in here are with respect to the mirror family of $D_0$. This is slightly different from the generalized functional invariants considered in Section \ref{sec:example-1} and Section \ref{sec:example-2} where the functional invariants are with respect to the mirror family of the smooth anticanonical divisor $D_{12}$ of $D_0$. In Proposition \ref{prop-toric-gen-func-inv} $q_i$ is considered as the generalized functional invariants for $D_0$ with respect to itself.
\end{remark}

We can also compute the generalized functional invariants when the base family is the mirror to the smooth anticanonical divisor $D_{12}$ of $D_0$. Recall that we have a Tyurin degeneration of $D$ into $D_1\cup_{D_{12}}\tilde D_2$, where $D_1\subset X_1$, $\tilde D_2\subset \tilde X_2$, induced by a refinement of nef partition. If we assume that $-K_Y-\sum_{l=0}^s c_1(L_l)$ is nef, then $D_{12}^\vee$ is defined by the fibrewise Calabi-Yau compactification of 
\begin{align}\label{mirror-D-12}
\left\{(x_1,\ldots,x_m)\in (\mathbb C^*)^m\left| \prod_{j=1}^m x_j^{m_{ij}}=q_i, i=1,\ldots,r;\right.\right. \\
\notag \left.\vphantom{\prod_{j=1}^m} F_l(x)=1,l=1,\ldots,s+1; F_{0,1}(x)=1, F_{0,2}(x)=1\right\}.
\end{align}
If we only assume that the restriction of $-K_Y-\sum_{l=0}^s c_1(L_l)$ is nef on a toric subvariety $Y^\prime \subset Y$, then $D_{12}^\vee$ is a specialization of (\ref{mirror-D-12}) to $q_i=0$ for some $i\in\{1,\ldots,r\}$. We will simply use (\ref{mirror-D-12}) to do the computation and we can specialize to $q_i=0$, for some $i\in\{1,\ldots,r\}$, when it is necessary. 

Then, we can compute generalized functional invariants with respect to $D_{12}^\vee$. Recall that the mirror for $D_0$ is the LG model 
\begin{align*}
W_0:D^\vee_0 &\rightarrow \mathbb C^*\\
(x_1,\ldots,x_m)&\mapsto F_{s+1}(x),
\end{align*}
where $D_{0}^\vee$ is the fibrewise Calabi-Yau compactification of 
\begin{align*}
\left\{(x_1,\ldots,x_m)\in (\mathbb C^*)^m\left| \prod_{j=1}^m x_j^{m_{ij}}=q_i, i=1,\ldots,r;\right.\right. \\
\left.\vphantom{\prod_{j=1}^m} F_l(x)=1,l=1,\ldots,s; F_{0,1}(x)=1, F_{0,2}(x)=1\right\}.
\end{align*}
Let $I_{s+1}$ be the index set for the summand of $F_{s+1}(x)$. We set $F_{s+1}(x)=y_0$ and consider the change of variables
\[
\bar x_j=\frac{x_j}{y_0}, j\in I_{s+1}
\]
and
\[
\lambda_{i,0}=\frac{q_{i}}{\prod_{j\in I_{s+1}}^{a+b}y_0^{m_{ij}}}, i=1\ldots, r.
\]
Generalized functional invariants of $X^\vee$, $X_1^\vee$ and $\tilde X_2^\vee$ with respect to $D_{12}^\vee$ can also be computed. The computation is similar to the computation for (\ref{gen-fun-inv-X}), (\ref{gen-fun-inv-X-1}) and (\ref{gen-fun-inv-X-2}) except that there is an extra change of variables  
\[
\bar x_j=\frac{x_j}{y_0}, j\in I_{s+1}.
\]
Therefore, the generalized functional invariants of $X^\vee$, $X_1^\vee$ and $\tilde X_2^\vee$ with respect to $D_{12}^\vee$ are
\[
\lambda_i^\prime=\lambda_i\frac{1}{\prod_{j\in I_{s+1}}^{a+b}y_0^{m_{ij}}}, i=1\ldots, r;
\]
\[
\lambda_{i,1}^\prime=\lambda_{i,1}\frac{1}{\prod_{j\in I_{s+1}}^{a+b}y_0^{m_{ij}}}, i=1\ldots, r;
\]
\[
\lambda_{i,2}^\prime=\lambda_{i,2}\frac{1}{\prod_{j\in I_{s+1}}^{a+b}y_0^{m_{ij}}}, i=1\ldots, r.
\]
Then, we again have the relation among generalized functional invariants
\begin{proposition}\label{prop-toric-gen-func-inv-'}
Under the identification (\ref{identification-X}), the generalized functional invariants satisfy the product relation
\[
\lambda_i^\prime\lambda_{i,0}=\lambda_{i,1}^\prime\lambda_{i,2}^\prime, i=1\ldots, r.
\]
\end{proposition}

\subsection{Periods}\label{sec:toric-period}

Recall that $p_1,\ldots,,p_r\in H^2(Y,\mathbb Z)$ is a nef integral basis and the toric divisors are
\[
\mathbf D_j=\sum_{i=1}^r m_{ij}p_i,\quad 1\leq j\leq m, 
\]
for some $m_{ij}$.

If we assume that $\rho_{s+1}=-K_Y-\sum_{l=0}^s c_1(L_l)$ is nef on $Y$, then the relative period for $(D_0^\vee,W_0)$ is
\begin{align}\label{period-D-0}
f_0^{D_0}(\lambda_0)=\sum_{\substack{d\in H_2(Y;\mathbb Z)\\ \forall j, \langle \mathbf D_j, d\rangle\geq 0}}\frac{\prod_{l=1}^s \langle \rho_l, d\rangle !}{\prod_{j=1}^m \langle \mathbf D_j, d\rangle !}\langle \rho_{0,1}, d \rangle !\langle \rho_{0,2}, d\rangle \langle \rho_{s+1},d\rangle! \lambda_0^d.
\end{align}
This relative period can be computed via the generalized functional invariant map as in \cite{DKY}*{Section 5}. This can also be obtained from the relative mirror theorem of \cite{FTY} when restricted to the toric case.

Suppose $\rho_{s+1}$ is nef when restricted to a toric subvariety $Y^\prime\subset Y$. Then the relative period for $(D_0^\vee, W_0)$ is a restriction of (\ref{period-D-0}) to $q_i=0$ for some of $i\in\{1,\ldots,r\}$.

We can compute the holomorphic period for $X^\vee$ via generalized functional invariants:
\begin{align*}
&\frac{1}{(1-y)}f_0^{D_0}(\lambda^\prime)\\
=&\sum_{\substack{d\in H_2(Y;\mathbb Z)\\ \forall j, \langle \mathbf D_j, d\rangle \geq 0}}\frac{\prod_{l=1}^{s+1} \langle \rho_l, d\rangle !}{\prod_{j=1}^m \langle \mathbf D_j, d\rangle !}\langle \rho_{0,1}, d\rangle !\langle \rho_{0,2}, d\rangle ! \langle \rho_{s+1},d\rangle!\frac{\lambda_0^d}{(1-y)^{1+\langle \rho_{0,1}, d\rangle }y^{\langle \rho_{0,2}, d\rangle}}\\
=&\sum_{\substack{d\in H_2(Y;\mathbb Z)\\ \forall j, \langle \mathbf D_j, d\rangle \geq 0}}\sum_{d_{0,1}\geq 0}\frac{\prod_{l=1}^{s+1} \langle \rho_l, d\rangle !}{\prod_{j=1}^m \langle \mathbf D_j, d\rangle !}\langle \rho_{0,1}, d\rangle !\langle \rho_{0,2}, d\rangle !\langle \rho_{s+1},d\rangle!\frac{(d_{0,1}+\langle \rho_{0,1}, d\rangle) !\lambda_0^d}{d_{0,1}!\langle \rho_{0,1}, d\rangle !y^{\langle \rho_{0,2}, d\rangle }}y^{d_{0,1}}.
\end{align*}
We obtain the holomorphic period for $X^\vee$:
\begin{align*}
f_0^{X}(\lambda_0)&=\frac{1}{2\pi i}\oint \frac{1}{(1-y)}f_0^{D_0}(\lambda^\prime)\frac{dy}{y}\\
&=\sum_{\substack{d\in H_2(Y;\mathbb Z)\\ \forall j, \langle \mathbf D_j, d\rangle \geq 0}}\frac{\prod_{l=1}^s \langle \rho_l, d\rangle !}{\prod_{j=1}^m \langle \mathbf D_j, d\rangle !}(\langle \rho_{0,1}, d\rangle+\langle \rho_{0,2}, d\rangle)!\langle \rho_{s+1},d\rangle!\lambda_0^d.\\
&=\sum_{\substack{d\in H_2(Y;\mathbb Z)\\ \forall j, \langle\mathbf  D_j, d\rangle \geq 0}}\frac{\prod_{l=0}^s \langle\rho_l, d\rangle !}{\prod_{j=1}^m \langle \mathbf D_j, d\rangle !}\langle \rho_{s+1},d\rangle!\lambda_0^d.
\end{align*}
The holomorphic periods for $X_1^\vee$ and $\tilde{X}_2^\vee$ are computed in a similar way, we have
\[
f_0^{X_1}(\lambda_0,y_1)=\sum_{\substack{d\in H_2(Y;\mathbb Z)\\ \forall j, \langle \mathbf D_j, d\rangle \geq 0}}\frac{\prod_{l=1}^s \langle \rho_l, d\rangle !}{\prod_{j=1}^m \langle \mathbf D_j, d\rangle !}\langle \rho_{0,1}, d\rangle !\langle \rho_{0,2}, d\rangle !\langle \rho_{s+1},d\rangle!\frac{\lambda_0^d}{(y_1)^{\langle \rho_{0,2}, d\rangle}};
\]
\[
f_0^{\tilde{X}_2}(\lambda_0,y)=\sum_{\substack{d\in H_2(Y;\mathbb Z)\\ \forall j, \langle \mathbf D_j, d\rangle \geq 0}}\sum_{d_{0}\geq 0}\frac{\prod_{l=1}^s \langle \rho_l, d\rangle !}{\prod_{j=1}^m \langle \mathbf D_j, d\rangle !}\frac{(d_{0}+\langle \rho_{0,1}, d\rangle) !\langle \rho_{0,2}, d\rangle !\lambda_0^d}{d_{0}!}\langle \rho_{s+1},d\rangle!(y)^{d_0}.
\]
Recall that $y=y_1=q_{0,2}/y_2$ and $\lambda_{i,0}=\frac{q_{i}}{\prod_{j\in I_{s+1}}^{a+b}y_0^{m_{ij}}}, i=1\ldots, r.$ Since $\lambda_0$ and $q$ are only differ by a scaling factor, we will simply replace $\lambda_0$ by $q$ for the rest of the section. Then we have 

\begin{theorem}\label{thm:gluing-period-general}
The following Hadamard product relation holds
\[
f_0^{X}(q)\star_q f_0^{D_0}(q)=\frac{1}{2\pi i}\oint f_0^{X_1}(q,y)\star_q f_0^{\tilde{X}_2}(q,y)\frac{dy}{y}.
\]
\end{theorem}

Similarly, we have the Hadamard product relation among Picard-Fuchs operators. Let
\begin{align*}
I^{X}(q)=e^{\sum_{i=0}^r p_i\log q_i}\sum_{d\in \on{NE}(Y)_{\mathbb Z}}&\prod_{j=1}^m \left(\frac{\prod_{k=-\infty}^0(\mathbf D_j+k)}{\prod_{k=-\infty}^{\langle \mathbf D_j,d\rangle}(\mathbf D_j+k)}\right)\left(\prod_{l=0}^s \prod_{k=1}^{\langle \rho_l, d\rangle }(\rho_l+k)\right)\prod_{k=1}^{\langle \rho_{s+1}, d\rangle }(\rho_{s+1}+k)q^d;
\end{align*}
\begin{align*}
I^{D_0}(q)=e^{\sum_{i=1}^r p_i\log q_i}\sum_{d\in \on{NE}(Y)_{\mathbb Z}}&\prod_{j=1}^m \left(\frac{\prod_{k=-\infty}^0(\mathbf D_j+k)}{\prod_{k=-\infty}^{\langle \mathbf D_j,d\rangle}(\mathbf D_j+k)}\right)\left(\prod_{l=1}^s \prod_{k=1}^{\langle \rho_l, d\rangle }(\rho_l+k)\right)\\
&\cdot\left(\prod_{k=1}^{\langle \rho_{0,1}, d\rangle }(\rho_{0,1}+k)\prod_{k=1}^{\langle \rho_{0,2}, d\rangle }(\rho_{0,2}+k)\right)\prod_{k=1}^{\langle \rho_{s+1}, d\rangle }(\rho_{s+1}+k)q^{d};
\end{align*}
\begin{align*}
I^{X_1}(q,y)=e^{\sum_{i=1}^r p_i\log \lambda_{i,1}}\sum_{d\in \on{NE}(Y)_{\mathbb Z},d_0\geq 0}&\prod_{j=1}^m \left(\frac{\prod_{k=-\infty}^0(\mathbf D_j+k)}{\prod_{k=-\infty}^{\langle \mathbf D_j,d\rangle}(\mathbf D_j+k)}\right)\left(\prod_{l=1}^s \prod_{k=1}^{\langle \rho_l, d\rangle }(\rho_l+k)\right)\\
&\cdot\left(\prod_{k=1}^{\langle \rho_{0,1}, d\rangle }(\rho_{0,1}+k)\prod_{k=1}^{\langle \rho_{0,2}, d\rangle }(\rho_{0,2}+k)\right)\prod_{k=1}^{\langle \rho_{s+1}, d\rangle }(\rho_{s+1}+k)\frac{q^d}{y^{\langle \rho_{0,2}, d\rangle }},
\end{align*}
where $\lambda_{i,1}=\frac{q_{i,1}}{\prod_{k=a+1}^{a+b}y_1^{m_{ij_k}}}, i=1\ldots, r$;

\begin{align*}
I^{\tilde X_2}(q,y)=e^{\sum_{i=1}^r p_i\log q_i+p_0\log y}&\sum_{d\in \on{NE}(Y)_{\mathbb Z},d_0\geq 0}\prod_{j=1}^m \left(\frac{\prod_{k=-\infty}^0(\mathbf D_j+k)}{\prod_{k=-\infty}^{\langle \mathbf D_j,d\rangle}(\mathbf D_j+k)}\right)\left(\prod_{l=1}^s \prod_{k=1}^{\langle \rho_l, d\rangle }(\rho_l+k)\right)\\
&\cdot\frac{\left(\prod_{k=1}^{\langle \rho_{0,1}, d\rangle +d_0}(\rho_{0,1}+p_0+k)\prod_{k=1}^{\langle \rho_{0,2}, d\rangle }(\rho_{0,2}+k)\right)}{\prod_{k=1}^{d_0}(p_0+k)}\prod_{k=1}^{\langle \rho_{s+1}, d\rangle }(\rho_{s+1}+k)y^{d_0}q^d.
\end{align*}
\begin{theorem}\label{thm:general-case-I-function}
The Hadamard product relation among the bases of solutions to Picard-Fuchs equations is
\[
I^{X}(q)\star_q I^{D_0}(q)=\frac{1}{2\pi i}\oint I^{X_1}(q,y)\star_q I^{\tilde{X}_2}(q,y)\frac{dy}{y},
\]
where we set $p_0=\rho_{0,2}$. In the Hadamard product $*_q$, we treat $\log q_i$ as a variable that is independent from $q_i$. Alternatively, we can consider $\bar{I}(q):=I(q)/(\sum p_i \log q_i)$ and write Hadamard product relation for $\bar{I}$.
\end{theorem}

\begin{remark}\label{rmk-periods}
The Hadamard product relation between periods is also true if we remove the factor $\langle \rho_{s+1},d\rangle!$ from each period. Similarly, we can remove the factor $\prod_{k=1}^{\langle \rho_{s+1}, d\rangle }(\rho_{s+1}+k)$ from $I$-functions appeared in Theorem \ref{thm:general-case-I-function}. Then the hypergeometric series are related to different enumerative invariants. For example, $I^{X}(q)$ is related to a generating function of genus zero relative Gromov-Witten invariants of $(X,D)$. Removing the factor $\prod_{k=1}^{\langle \rho_{s+1}, d\rangle }(\rho_{s+1}+k)$ from $I^{X}(q)$, it is related to a generating function of genus zero absolute Gromov-Witten invariants $X$ instead of relative Gromov-Witten invariants of $(X,D)$.
\end{remark}

\section{Degeneration to the normal cones}\label{sec:deg-normal-cone}
In this section, we consider a special kind of degeneration: degeneration to the normal cones. Let $X$ be a quasi-Fano variety with its smooth anticanonical divisor $D$. We can consider the degeneration to the normal cone of $X$ with respect to a divisor $D_1$. More precisely, let $\mathcal X$ be the blow-up of $X\times \mathbb A^1$ along the subvariety $D_1\times \{0\}$. There is a projection map $\pi:\mathcal X\rightarrow \mathbb A^1$ to the second factor. The central fibre $\pi^{-1}(0)$ is $X\cup_{D_1}\mathbb P_{D_1}(N_{D_1}\oplus \mathcal O)$.

Let $D_1+D_2\in |-K_X|$ be simple normal crossings and $D_{12}=D_1\cap D_2$. Then, the degeneration of $X$ also gives a Tyurin degeneration of $D$ into $D_2\subset X$ and a blow-up $\tilde D_1$ of $D_1$. The blown-up variety $\tilde D_1$ is in $\mathbb P_{D_1}(N_{D_1}\oplus \mathcal O)$. Therefore, the degeneration that we are considering is the following:
\[
(X,D)\leadsto (X,D_2)\cup_{(D_1,D_{12})}(\mathbb P_{D_1}(N_{D_1}\oplus \mathcal O), \tilde D_1).
\]

Then Conjecture \ref{conj-gluing-1} states that the LG mirror of $(X,D)$ can be obtained by gluing rank $2$ LG mirrors of $(X,D_1+D_2)$ and $(\mathbb P_{D_1}(N_{D_1}\oplus \mathcal O),D_1+\tilde D_1)$. Recall that, suppose $(X^\vee,h=(h_1,h_2))$ is the rank $2$ LG mirror of $(X,D_1+D_2)$, then $W=h_1+h_2: X^\vee\rightarrow \mathbb C$ is a non-proper LG model. By mirror symmetry, it is expected that the LG model $(\bar {X^\vee},\bar{W})$ of $(X,D)$ is a partial (i.e. fibrewise) compactification of the non-proper LG model $(X^\vee,W)$. Our conjecture suggests the following.
\begin{conjecture}\label{conj-compactification}
Given a quasi-Fano variety $X$ with its smooth anticanonical divisor $D$. Let $D_1+D_2$ be a simple normal crossings anticanonical divisor of $X$. The ordinary LG model of $(X,D)$ is a (partial) compactification of the rank $2$ LG model of $(X,D_1+D_2)$ and the compactification can be obtained by gluing the rank $2$ LG mirror of $(\mathbb P_{D_1}(N_{D_1}\oplus \mathcal O),D_1+\tilde D_1)$ to the rank $2$ LG mirror of $(X,D_1+D_2)$.
\end{conjecture}

When $D\subset X$ is a toric complete intersection of a toric variety determined by a nef partition. The refinement of the nef partition discussed in \cite{DHT}*{Section 3} gives a degeneration of $(X,D)$ which is compatible with a Tyurin degeneration of $D$. One can obtain the relation between their generalized functional invariants, periods, and $I$-functions following the procedure of \cite{DKY}*{Section 5}. This is similar to the discussion in Section \ref{sec:toric}, but slightly different. Suppose $X$ is a complete intersection in a toric variety $Y$ defined by a generic section of $E=L_0\oplus L_1\oplus\cdots \oplus  L_s$ and $D\subset X$ is a Calabi-Yau intersection in $Y$ defined by a generic section of $E=L_0\oplus L_1\oplus\cdots \oplus  L_s\oplus L_{s+1}$. For degeneration to the normal cones, we consider a refinement of the nef partition with respect to the part $L_{s+1}$. While, in Section \ref{sec:toric}, we consider a refinement of the nef partition with respect to $L_i$ for $i\in \{0,1,\ldots,s\}$. We will work out an example of the degeneration to the normal cone explicitly in the next section. The general case works similarly.

\subsection{An example}\label{sec:ex-deg-normal-cone}
In this section, we will consider an example of the degeneration to the normal cone. We consider a degeneration to the normal cone of $\mathbb P^3$ with respect to a smooth cubic surface $S$:
\[
\mathbb P^3\leadsto \mathbb P^3\cup_S \mathbb P(N_S\oplus O_S).
\]
Then a quartic $K3$ surface in $\mathbb P^3$ degenerates to $\mathbb P^2\subset \mathbb P^3$ and a blow-up of $S$ along a complete intersection of a degree one and a degree four hypersurfaces. This blown-up variety $\tilde D$ is a hypersurface in $Y:=\mathbb P(N_S\oplus O_S)$ defined by the vanishing locus of a generic section of $\mathcal O_Y(1)\otimes\pi^* \mathcal O_S(1)$, where $\pi: Y \rightarrow S$ is the projection.

Therefore, what we consider is the degeneration of
$(\mathbb P^3,K3)$ into $(\mathbb P^3,\mathbb P^2)$ and $(\mathbb P(N_S\oplus O_S),\tilde D)$ intersecting along $S$. We study the relation among their mirrors. The mirror of $(\mathbb P^3,K3)$ is an LG model given by the fibrewise compactification of
\begin{align*}
W: (\mathbb C^*)^3&\rightarrow \mathbb C\\
 (x_0,x_1,x_2)&\mapsto x_0+x_1+x_2+\frac {q_1}{x_0x_1x_2}.
\end{align*}
It can be rewritten as
\begin{align*}
W:X^\vee\rightarrow \mathbb C\\
(x_0,x_1,x_2,y)\mapsto y,
\end{align*}
where $X^\vee$ is the fibrewise compactification of
\[
\left\{ (x_0,x_1,x_2,y_1)\in (\mathbb C^*)^4\left| x_0+x_1+x_2+\frac{q_{1}}{x_0x_1x_2}=y\right. \right\}.
\]

The rank $2$ LG model mirrors to $(\mathbb P^3,\mathbb P^2+S)$ is given by a fibrewise compactification of
\begin{align*}
h_{1}:=(h_{11},h_{12}): (\mathbb C^*)^3&\rightarrow (\mathbb C)^2\\
 (x_0,x_1,x_2)&\mapsto (x_0, y_1:=\frac {q_{1,1}}{x_0x_1x_2}+x_1+x_2).
\end{align*}
The rank $2$ LG model mirrors to $(\mathbb P(N_S\oplus O_S),\tilde D+S)$ is given by
\begin{align*}
h_2:X_2^\vee\rightarrow (\mathbb C)^2\\
(x_0,x_1,x_2,y,y_2)\mapsto (y,y_2),
\end{align*}
where $X^\vee_2$ is the fibrewise compactification of
\[
\left\{ (x_0,x_1,x_2,y,y_2)\in (\mathbb C^*)^5\left| x_1+x_2+\frac{q_{1,2}}{x_0x_1x_2(y-x_0)^3}=1, (y-x_0)y_2=q_{2,2}\right. \right\}.
\]

Recall that the mirror LG model of the cubic surface $D_0$ is defined by
\begin{align*}
W_1:D_0^\vee\rightarrow \mathbb C\\
(x_0,x_1,x_2)\mapsto x_0,
\end{align*}
where $D_0^\vee$ is the fibrewise compactification of
\[
\left\{(x_0,x_1,x_2)\in (\mathbb C^*)^3\left| x_1+x_2+\frac{q_{1,0}}{x_0x_1x_2}=1\right.\right\}.
\]
The functional invariants are the following:
\begin{align}\label{func-inv-normal-cone}
\lambda=\frac{q_1}{x_0(y-x_0)^3}, \quad \lambda_0=\frac{q_{1,0}}{x_0}, \quad \lambda_1=\frac{q_{1,1}}{x_0y_1^3}, \quad \lambda_2=\frac{q_{1,2}y_2^3}{(y-q_{2,2}/y_2)q_{2,2}^3}. 
\end{align}
We have the following identification of variables
\begin{align}\label{identification-deg-normal-cone}
q_1=q_{1,0}=q_{1,1}=\frac{q_{1,2}y_2^3}{q_{2,2}^3}, \quad y_1=q_{2,2}/y_2=y-x_0.
\end{align}
The relation among their functional invariants are the following
\[
\lambda\cdot \lambda_0=\lambda_1\cdot \lambda_2.
\]

The relative period for $(S^\vee, W_0)$ is
\[
f_0^{S^\vee}(q_{1,0},x_0)=\sum_{d_1\geq 0}(q_{1,0}/x_0)^{d_1}\frac{(3d_1)!}{(d_1!)^3}.
\]
The relative period of $(\mathbb P^3,K3)$ can be computed as a residue integral of the pullback of
\[
\frac{1}{(1-x_0/y)}f_0^{E^\vee}(\tilde\lambda).
\]
by the functional invariant
\[
\lambda=\frac{q_1}{x_0(y-x_0)^3}=\frac{q_1}{x_0y^3}\frac{1}{(1-x_0/y)^3}.
\]
Therefore, we have
\begin{align*}
\frac{1}{(1-x_0/y)}f_0^{E^\vee}(\lambda)&=\frac{1}{(1-x_0/y)}\sum_{d_1\geq 0}\frac{(3d_1)!}{(d_1!)^3}\left(\frac{q_1}{x_0y^3}\frac{1}{(1-x_0/y)^3}\right)^{d_1}\\
&=\sum_{d_1,d_{0,1}\geq 0}\frac{(3d_1)!}{(d_1!)^3}\left(\frac{q_1}{x_0y^3}\right)^{d_1}\frac{(3d_1+d_{0,1})!}{(3d_1)!d_{0,1}!}(x_0/y)^{d_{0,1}}\\
&=\sum_{d_1,d_{0,1}\geq 0}\frac{(3d_1+d_{0,1})!}{(d_1!)^3d_{0,1}!}\left(\frac{q_1}{x_0y^3}\right)^{d_1}\left(\frac{x_0}{y}\right)^{d_{0,1}}
\end{align*}

After taking a residue, we obtain the holomorphic period of $(\mathbb P^3,K3)$: 
\begin{align*}
 f_0^{X^\vee}(q_1,y)&=\frac{1}{2\pi i}\oint \frac{1}{(1-x_0/y)}f_0^{E^\vee}(\lambda) \frac{dx_0}{x_0}\\
&=\sum_{d_1\geq 0}\frac{(3d_1+d_1)!}{(d_1!)^3d_1!}\left(\frac{q_1}{y^3}\right)^{d_1}\left(\frac{1}{y}\right)^{d_{1}}\\
&=\sum_{d_1\geq 0}\frac{(4d_1)!}{(d_1!)^4}\left(\frac{q_1}{y^4}\right)^{d_1}.\\
\end{align*}
This again matches with the previous computation of relative periods in \cite{DKY}.

Relative period for the rank $2$ LG model of $(\mathbb P^3,\mathbb P^2+S)$ is
\begin{align*}
f_0^{X_1^\vee}(q_{1,1},y_1,x_0)=\sum_{d_1\geq 0}\left(\frac{q_{1,1}}{x_0y_1^3}\right)^{d_1}\frac{(3d_1)!}{(d_1!)^3}.
\end{align*}

Now, we compute the relative period for the rank $2$ LG model of $(\mathbb P(N_S\oplus O_S),\tilde D+S)$. Recall that the functional invariant is
\[
\lambda_2=\frac{q_{1,2}y_2^3}{(y-q_{2,2}/y_2)q_{2,2}^3}=\frac{q_{1,2}y_2^3}{yq_{2,2}^3}\frac{1}{1-q_{2,2}/(y_2y)}.
\]
The relative period is
\begin{align*}
&\frac{1}{1-q_{2,2}/(y_2y)}f_0^{E^\vee}(\lambda_2)\\
=&\frac{1}{1-q_{2,2}/(y_2y)} \sum_{d_1\geq 0}\left(\frac{q_{1,2}y_2^3}{yq_{2,2}^3}\right)^{d_1}\left(\frac{1}{1-q_{2,2}/(y_2y)}\right)^{d_1}\frac{(3d_1)!}{(d_1!)^3}\\
=&\sum_{d_1,d_{0,2}\geq 0}\left(\frac{q_{1,2}y_2^3}{yq_{2,2}^3}\right)^{d_1}\frac{(d_1+d_{0,2})!}{d_1!d_{0,2}!}\left(\frac{q_{2,2}}{y_2y}\right)^{d_{0,2}}\frac{(3d_1)!}{(d_1!)^3}.
\end{align*}
Under the identification (\ref{identification-deg-normal-cone}), the relative period can be rewritten as
\[
f_0^{X_2^\vee}(q_1,y_1,y)=\sum_{d_1,d_{0,2}\geq 0}\left(\frac{q_1}{y}\right)^{d_1}\frac{(d_1+d_{0,2})!}{d_1!d_{0,2}!}\left(\frac{y_1}{y}\right)^{d_{0,2}}\frac{(3d_1)!}{(d_1!)^3}.
\]
Then we can glue periods as follows
\begin{align}\label{gluing-period-degen-normal-cone}
\notag    &\frac{1}{2\pi i}\oint f_0^{X_1^\vee}(q_1,y_1,x_0)\star_{q_1} f_0^{X_2^\vee}(q_1,y_1,y)\frac{dy_1}{y_1}\\
 \notag   =&\frac{1}{2\pi i}\oint \left(\sum_{d_1\geq 0}\left(\frac{q_1}{x_0y_1^3}\right)^{d_1}\frac{(3d_1)!}{(d_1!)^3}\right)\star_{q_1} \left(\sum_{d_1,d_{0,2}\geq 0}\left(\frac{q_1}{y}\right)^{d_1}\frac{(d_1+d_{0,2})!}{d_1!d_{0,2}!}\left(\frac{y_1}{y}\right)^{d_{0,2}}\frac{(3d_1)!}{(d_1!)^3}\right)\frac{dy_1}{y_1}\\
\notag    =&\sum_{d_1\geq 0}\left(\frac{q_1}{x_0}\right)^{d_1}\frac{(3d_1)!}{(d_1!)^3} \left(\frac{q_1}{y}\right)^{d_1}\frac{(d_1+3d_1)!}{d_1!(3d_1)!}\left(\frac{1}{y}\right)^{3d_1}\frac{(3d_1)!}{(d_1!)^3}\\
\notag    =&\left(\sum_{d_1\geq 0}\frac{(4d_1)!}{(d_1!)^4}\left(\frac{q_1}{y^4}\right)^{d_1}\right)\star_{q_1} \left(\sum_{d_1\geq 0}\left(\frac{q_1}{x_0}\right)^{d_1}\frac{(3d_1)!}{(d_1!)^3}\right)\\
    =& f_0^{X^\vee}(q_1,y) \star_{q_1} f_0^{S^\vee}(q_1,x_0).
\end{align}

Just like previous examples, these periods are solutions to systems of PDEs. We can write down the bases of solutions for the systems of PDEs and they also satisfy the Hadamard product relation.
\section{Iterating the Doran-Harder-Thompson conjecture}\label{sec:iterating-DHT}

In this section, we apply the conjecture of gluing rank $2$ LG models in Section \ref{sec:gluing} to iterate the Doran-Harder-Thompson conjecture. 

We consider a Tyurin degeneration of a Calabi-Yau variety $X$ such that the central fibre of the degeneration is a simple normal crossing variety $X_1\cup_D X_2$ where $X_1$ and $X_2$ are quasi-Fano varieties and $D$ is their common smooth anti-canonical divisor. We further degenerate the pairs $(X_1,D)$ and $(X_2,D)$ into $(X_{11},D_{1})\cup_{D_{10}}(X_{12},D_{2})$ and $(X_{21},D_{1})\cup_{D_{20}}(X_{22},D_{2})$ respectively, such that the degeneration of $D$ into $D_1\cup_{D_1\cap D_2}D_2$ is a Tyurin degeneration. Then by \cite{DHT} and Conjecture \ref{conj-gluing-1}, we expect that the mirror of $X$ can be obtained by gluing four rank $2$ LG models.

Suppose $X$ is a Calabi-Yau complete intersection in a toric variety given by a nef partition. A refinement of the nef partition gives a Tyurin degeneration of $X$ into a union of quasi-Fano varieties $X_1$ and $X_2$ meeting normally along $D$. Following \cite{DKY}*{Section 5}, $X_1$, $X_2$ and $D$ are still complete intersections in certain toric varieties. We can further degenerate $X_1$ and $X_2$ following Section \ref{sec:toric}. Then $X_{11}, X_{12}, X_{21}$ and $X_{22}$ are all toric complete intersections. The relation among their generalized functional invariants follows from the relation that we obtained in \cite{DKY}*{Section 5} and Section \ref{sec:toric}. We explain it here in full detail for the degeneration of the quintic threefold.

We consider the degeneration of the quintic threefold $Q_5\subset \mathbb P^4$ into $(\mathbb P^3,K3)$ and $(\on{Bl}_{C_{1,5}}Q_4,K3)$, where $\on{Bl}_{C_{1,5}}Q_4$ is the blow-up of a quartic fourfold $Q_4$ along a complete intersection of degree $1$ and $5$ hypersurfaces. Then we further degenerate $(\mathbb P^3,K3)$ into $(\mathbb P^3,\mathbb P^2)$ and $(\mathbb P(N_S\oplus \mathcal O),\tilde D)$ intersecting along a cubic surface $S_1$ which is described in Section \ref{sec:ex-deg-normal-cone}. We also degenerate $(\on{Bl}_{C_{1,5}}Q_4,K3)$ into $(\on{Bl}_{C_{1,5}}\mathbb P^3,\mathbb P^2)$ and $(\on{Bl}_{C_{1,4}}\on{Bl}_{C_{1,5}}Q_3,\tilde D)$ intersecting along a common divisor $S_2$. 

The blown-up varieties can be realized as complete intersections in toric varieties. The variety $\on{Bl}_{C_{1,5}}Q_4$ has been studied in \cite{DKY}*{Section 4.2}. It is a complete intersection of degrees $(1,1)$ and $(4,0)$ in the toric variety $\mathbb P(\mathcal O_{\mathbb P^4}(-4)\oplus \mathcal O_{\mathbb{P}^4})$. The variety $\on{Bl}_{C_{1,5}}\mathbb P^3$ can be constructed similarly. It is a hypersurface of degree $(1,1)$ in the toric variety $\mathbb P(\mathcal O_{\mathbb P^3}(-4)\oplus \mathcal O_{\mathbb{P}^3})$. Similarly, $\on{Bl}_{C_{1,4}}\on{Bl}_{C_{1,5}}Q_3$ can be considered as follows. Let $X^\prime=\on{Bl}_{C_{1,5}}Q_3$. Then $X^\prime$ is a complete intersection given by the zero locus of generic sections of lines bundles  $\mathcal O_Y(1)\otimes \pi^*\mathcal O_{\mathbb P^4}(1)$ and $\pi^*\mathcal O_{\mathbb P^4}(3)$ in the toric variety 
\[
\pi:Y:=\mathbb P(\mathcal O_{\mathbb P^4}(-4)\oplus \mathcal O_{\mathbb{P}^4})\rightarrow \mathbb P^4.
\]
Then $\on{Bl}_{C_{1,4}}\on{Bl}_{C_{1,5}}Q_3=\on{Bl}_{C_{1,4}}X^\prime$ is a complete intersection given by the zero locus of generic sections of lines bundles $\mathcal O_{Y^\prime}(1)\otimes (\pi^{\prime})^*\pi^*\mathcal O_{\mathbb P^4}(1)$, $(\pi^{\prime})^*\mathcal O_Y(1)\otimes (\pi^{\prime})^*\pi^*\mathcal O_{\mathbb P^4}(1)$ and $(\pi^{\prime})^*\pi^*\mathcal O_{\mathbb P^4}(3)$  in 
\[
\pi^\prime:Y^\prime=\mathbb P_{Y}(\pi^*\mathcal O_{\mathbb P^4}(-3)\oplus \mathcal O_Y)\rightarrow Y.
\]

The mirrors can be written down following Section \ref{sec:toric}.
Recall that the functional invariant for the mirror of $(\mathbb P^3,K3)$ with respect to the mirror cubic curve family is 
\[
\lambda_1=\frac{q_{1,1}}{x_0(y-x_0)^3}
\]
as in (\ref{func-inv-normal-cone}). We can compute the functional invariants for mirrors of $Q_5$, $(\on{Bl}_{C_{1,4}}Q_4, K3)$ and the quartic $K3$:
\[
\lambda=\frac{q_1}{x_0(1-y)(y-x_0)^3}, \lambda_2=\frac{q_{1,2}y_2^4}{x_0(1-x_0)^3(1-q_2/y_2)}, \lambda_0=\frac{q_{1,0}}{x_0(1-x_0)^3}. 
\]
Under the identifications
\[
y=q_2/y_2, \quad \text{and } q_1=q_{1,1}=q_{1,2}y_2^4=q_{1,0},
\]
we have the product relation among functional invariants
\[
\lambda\lambda_0=\lambda_1\lambda_2.
\]

\begin{remark}
Note that in our earlier paper \cite{DKY}, we explained how to glue two functional invariants $\lambda_1=\frac{Q_1}{(y-1)^4}$ and $\lambda_2=\frac{Q_2}{y}$ to get $\lambda=\frac{Q}{y(y-1)^4}$. By replacing $Q_i$ with the expression $\lambda_0$ above, we obtain the expressions $\lambda_1,\lambda_2$. 
In this sense, the gluing formula above is really the same as starting with our old one and then replacing the parameters with $\lambda_0$, as was described above in Section \ref{sec:toric}.
\end{remark}

Here is a brief analysis of the singular loci for the mirror cubic elliptic curve fibration of the mirror quintic threefold. 
Recall once more that the mirror cubic family of elliptic curves has singular fibres of types $\textrm{I}_3,\textrm{I}_1,\textrm{IV}^*$ over $0,\frac{1}{3^3},\infty$ respectively. 

\textbf{Functional invariant $\lambda_1$:} 
\begin{itemize}
    \item $\lambda_1^{-1}(0)=\{x_0=\infty\}\cup\{(y-x_0)^3=\infty\}$ supports $\textrm{I}_3$ and $\textrm{I}_9$ singular fibres respectively.
    \item $\lambda_1^{-1}(\infty)=\{x_0=0\}\cup\{(y-x_0)^3=0\}$; the first component supports $\textrm{IV}^*$-fibres with the second component supporting smooth fibres. 
    \item The locus $\lambda_1^{-1}(\frac{1}{3^3})=\{x_0(y-x_0)^3=3^3q_1\}$ supports $\textrm{I}_1$-fibres generically. 
\end{itemize} 

\textbf{Functional invariant $\lambda_2$:} 
\begin{itemize}
    \item $\lambda_2^{-1}(0)=\{x_0=\infty\}\cup\{y=1\}\cup\{(1-x_0)^3=\infty\}$ with the first two components supporting $\textrm{I}_3$-fibres and the last supporting $\textrm{I}_9$-fibres. 
    \item $\lambda_2^{-1}(\infty)=\{x_0=0\}\cup\{(1-y)=\infty\}\cup\{(1-x_0)^3=0\}$. The first two  components support $\textrm{IV}^*$-fibres with the third component supporting smooth fibres. 
    \item The locus $\lambda_2^{-1}(\frac{1}{3^3})=\{x_0(1-x_0)^3)(1-y)=3^3q_1\}$ supports $\textrm{I}_1$-fibres.
\end{itemize}

\textbf{Functional invariant $\lambda_0$:} 
\begin{itemize}
    \item $\lambda_0^{-1}(0)=\{x_0=\infty\}\cup\{(1-x_0)^3=\infty\}$ supports $\textrm{I}_3$ and $\textrm{I}_9$ singular fibres respectively. 
    \item $\lambda_0^{-1}(\infty)=\{x_0=0\}\cup\{(1-x_0)^3=0\}$. The first component supports $\textrm{IV}^*$-fibres with the second component supporting smooth fibres.
    \item The locus $\lambda_0^{-1}(\frac{1}{3^3})=\{x_0(1-x_0)^3=3^3q_1\}$ supports $\textrm{I}_1$-fibres. 
\end{itemize}

\textbf{Functional invariant $\lambda$:}  
\begin{itemize}
    \item $\lambda^{-1}(0)=\{x_0=\infty\}\cup\{(1-y)=\infty\}\cup\{(y-x_0)^3=\infty\}$, with the first two components supporting $\textrm{I}_3$ fibres and the last component supporting $\textrm{I}_9$-fibres. 
    \item $\lambda^{-1}(\infty)=\{x_0=0\}\cup\{(1-y)=0\}\cup\{(y-x_0)^3=0\}$; the first two component supporting $\textrm{IV}^*$-fibres with the last component supporting smooth fibres. 
    \item The locus $\lambda^{-1}(\frac{1}{3^3})=\{x_0(1-y)(y-x_0)^3=3^3q_1\}$ supports $\textrm{I}_1$-fibres. 
\end{itemize}

The singular loci of $\lambda$ are essentially the union of the singular loci of $\lambda_1,\lambda_2$ with the exception of the removal of one factor of $x_0=0$ and $x_0=\infty$ that is accounted for in $\lambda_0$.

Now we further degenerate $(\mathbb P^3,K3)$ and $(\on{Bl}_{C_{1,4}}Q_4, K3)$. By Proposition \ref{prop-toric-gen-func-inv} and (\ref{func-inv-normal-cone}), the functional invariants $\lambda_1$ and $\lambda_2$ satisfy the following relations:
\[
\lambda_1\lambda_{1,0}=\lambda_{1,1}\lambda_{1,2}
\]
and
\[
\lambda_2\lambda_{2,0}=\lambda_{2,1}\lambda_{2,2},
\]
where $\lambda_{i,j}$ are functional invariants of degenerated pieces (and the intersections of degenerated pieces) with respect to the mirror cubic family; the identities are obtained under appropriate identifications of variables.
Therefore, we have the following relation among functional invariants
\begin{align}\label{identity-func-inv-iterative}
\lambda\lambda_0\lambda_{1,0}\lambda_{2,0}=\lambda_{1,1}\lambda_{1,2}\lambda_{2,1}\lambda_{2,2}.
\end{align}

Now, we would like to show that the period for the mirror quintic can be obtained by gluing relative periods of four rank $2$ LG models.
As computed in \cite{DKY}*{Section 2.2.1}, the period of the mirror quintic can be obtained from the residue integral of the pullback of the period of its internal fibration of mirror quartic $K3$ by the generalized functional invariant $\lambda=\frac{q}{y(1-y)^4}$. Similarly, we can compute the period of the mirror quintic by taking double residue integrals of the pullback of the period of its internal elliptic fibration by the functional invariant $\lambda=\frac{q_1}{x_0(1-y)(y-x_0)^3}$. The computation is as follows. Recall that, the holomorphic period for the mirror cubic curve family is
\[
f_0^{E^\vee}(\tilde \lambda)=\sum \frac{(3d)!}{(d!)^3}\tilde{\lambda}^d.
\]
Then
\begin{align*}
\frac{1}{(1-y)(1-x_0/y)}f_0^{E^\vee}(\lambda)&=\frac{1}{(1-y)(1-x_0/y)}\sum_{d_1\geq 0}\frac{(3d_1)!}{(d_1!)^3}\left(\frac{q_1}{x_0y^3(1-y)(1-x_0/y)^3}\right)^{d_1}\\
&=\frac{1}{(1-x_0/y)}\sum_{d_1,d_{0,1}\geq 0}\frac{(3d_1)!}{(d_1!)^3}\left(\frac{q_1}{x_0y^3(1-x_0/y)^3}\right)^{d_1}\frac{(d_1+d_{0,1})!}{d_1!d_{0,1}!}(y)^{d_{0,1}}\\
&=\sum_{d_1,d_{0,1},d_{0,2}\geq 0}\frac{(3d_1)!}{(d_1!)^3}\left(\frac{q_1}{x_0y^3}\right)^{d_1}\frac{(d_1+d_{0,1})!}{d_1!d_{0,1}!}\frac{(3d_1+d_{0,2})!}{(3d_1)!d_{0,2}!}(y)^{d_{0,1}}\left(\frac{x_0}{y}\right)^{d_{0,2}}.
\end{align*}
We take the residue integrals to obtain the holomorphic period of the mirror of $X=Q_5$:
\begin{align*}
\notag f_0^{X^\vee}(q_1)&=\frac{1}{(2\pi i)^2}\oint\oint \frac{1}{(1-y)(1-x_0/y)}f_0^{E^\vee}(\lambda) \frac{dy}{y}\frac{dx_0}{x_0}\\
\notag &=\sum_{d_1,d_0\geq 0}\frac{(3d_1)!}{(d_1!)^3}\left(q_1\right)^{d_1}\frac{(d_1+4d_{1})!}{d_1!(4d_{1})!}\frac{(3d_1+d_{1})!}{(3d_1)!d_{1}!}\\
&=\sum_{d_1\geq 0}\frac{(5d_1)!}{(d_1!)^5}q_1^{d_1}.
\end{align*}
Similarly, we can compute relative periods for the mirrors of $(X_1,D):=(\mathbb P^3,K3)$, $(X_2,D):=(\on{Bl}_{C_{1,4}}Q_4, K3)$ and the quartic $D:=K3$ as
\[
f_0^{\tilde X_1^\vee}(q_1,y)=\frac{1}{2\pi i}\oint \frac{1}{(1-x_0/y)}f_0^{E^\vee}(\lambda_1) \frac{dx_0}{x_0},
\]
\[
f_0^{\tilde X_2^\vee}(q_1,y)=\frac{1}{2\pi i}\oint \frac{1}{(1-y)(1-x_0)}f_0^{E^\vee}(\lambda_2) \frac{dx_0}{x_0}
\]
\[
f_0^{\tilde D^\vee}(q_1)=\frac{1}{2\pi i}\oint \frac{1}{(1-x_0)}f_0^{E^\vee}(\lambda_0) \frac{dx_0}{x_0}.
\]
We have
\begin{align}\label{identity-iterating-Q-5}
f_0^{X^\vee}(q_1)\star_{q_1}f_0^{\tilde D^\vee}(q_1)=\frac{1}{2\pi i}\oint f_0^{\tilde X_1^\vee}(q_1,y)\star_{q_1}f_0^{\tilde X_2^\vee}(q_1,y)\frac{dy}{y}.
\end{align}
On the other hand, $f_0^{\tilde X_1^\vee}(q_1,y)$ and $f_0^{\tilde X_2^\vee}(q_1,y)$ satisfy the following product relations by Theorem \ref{thm:gluing-period-general} and Equation (\ref{gluing-period-degen-normal-cone}):
\begin{align*}
& f_0^{X^\vee}\star_{q_1}f_0^{\tilde D^\vee}\star_{q_1}f_0^{D_{10}^\vee}\star_{q_1}f_0^{D_{20}^\vee}\\
=& \frac{1}{(2\pi i)^2}\oint \left(\left(\oint f_0^{\tilde X_{11}^\vee}\star_{q_1}f_0^{\tilde X_{12}^\vee}\frac{dy_1}{y_1}\right)\star_{q_1}\left(\oint f_0^{\tilde X_{21}^\vee}\star_{q_1}f_0^{\tilde X_{22}^\vee}\frac{dy_2}{y_2}\right)\right)\frac{dy}{y}.
\end{align*}

These holomorphic periods are solutions to systems of PDEs. This relation between periods is also true for the bases of solutions for the systems of PDEs. In general, such a relation is true when we consider iterating Tyurin degenerations of Calabi-Yau complete intersections in toric varieties. 

In the above example of quintic, the degeneration is induced from a refinement of the nef partition $(5)\to (1)+(4)$. There is another refinement of the nef partition that one can consider: $(5)\to (2)+(3)$. In this case, the quintic threefold is degenerated into a blow-up $\on{Bl}_{3,5} Q_2$ of quadric threefold $Q_2$ along complete intersection of degree $3$ and $5$ hypersurfaces and a cubic threefold $Q_3$. The intersection of $\on{Bl}_{3,5} Q_2$ and $ Q_3$ is a sextic $K3$ surface in $\mathbb P^4$ whose mirror is a family of $M_3$-polarized $K3$ surfaces. The family of $M_3$-polarized $K3$ surfaces also admits an internal fibration structure over $\mathbb P^1$ whose generic fiber is a mirror cubic curve. One can compute the functional invariants of the mirrors of $Q_5$, $(\on{Bl}_{3,5} Q_2,K3)$, $(Q_3,K3)$ and the sextic $K3$ with respect to the mirror cubic family. The computation is similar to the previous example and we have a product relation among functional invariants
\[
\lambda\lambda_0=\lambda_1\lambda_2.
\]
We can further degenerate $(\on{Bl}_{3,5} Q_2,K3)$ and $(Q_3,K3)$. We degenerate $(\on{Bl}_{3,5} Q_2,K3)$  into $(\on{Bl}_{3,5} \mathbb P^3,C)$ and $(\on{Bl}_{1,2}\on{Bl}_{3,5}\mathbb P^3,\tilde C)$ where $C$ is a cubic surface and $\tilde C$ is a blow-up of a cubic surface. We degenerate $(Q_3,K3)$ into $(Q_3,C)$ and $(\mathbb P_{C}(N_C\oplus O_C), \tilde C)$. We then also have the relation for functional invariants of the two-step degenerations which takes the same form as Identity (\ref{identity-func-inv-iterative}).

\begin{remark}
It is natural to ask: viewing $Q_5^\vee$ as a family of elliptic covers over $\mathbb{P}^1\times\mathbb{P}^1$, what are the families of rational curves on the base that support $K3$ surfaces? The $x_0$ and $(y-1)$ slices give rise to families of $M_3$-polarized K3 surfaces and we obtain $M_2$-polarized $K3$ surfaces by considering $(y-x_0)$-slices. 
Are there others? 

The answer is no, and here is one way to see this.
Firstly, any rational curve in $\mathbb{P}^1\times\mathbb{P}^1$ is a curve of type $(1,d)$ or $(d,1)$. 
That is, in an affine chart, it can be represented as the graph of a rational function given by the ratio of two degree $d$-polynomials. 
In turn, by replacing $(x_0,1-y)$ with the corresponding expressions in $\lambda$ above, we obtain the generalized functional invariant of the corresponding (family of) elliptic fibration(s). 
Unless $d=0$ or $d=1$, the degree of the resulting functional invariant will be at least $5$. 
On the other hand, Riemann-Hurwitz and fibre-type considerations readily imply that the maximum degree generalized functional invariant for a genus $0$-cover of the $\lambda$-line that gives rise to a $K3$ surface is $4$.
\end{remark}

The above degenerations of quintic threefold are mirror to the elliptic fibration of the mirror quintic over the $\mathbb P^1\times \mathbb P^1$ base. In general, one can expect the following generalization of the Doran-Harder-Thompson conjecture, which is also mentioned in \cite{BD} and \cite{DT},
\begin{conjecture}\label{conj-higher-dim-base}
If a Calabi-Yau variety $X$ (or a log Calabi-Yau variety $X\setminus D$) admits a semi-stable degeneration, connected to a point of maximal unipotent monodromy, such that the cone over the dual intersection complex of the central fibre is of dimension $d$, then the mirror of $X$ (or $X\setminus D$) admits a Calabi-Yau fibration structure with a $d$-dimensional base. 
\end{conjecture}

There are already some evidence of this conjecture. Given a semi-stable degeneration, one can consider the Clemens-Schmid exact sequence relating the geometry of the central fibre of the degeneration and the geometry of a nearby smooth fibre. On the mirror side, if a smooth variety  admits a projective fibration over a projective base, Doran-Thompson \cite{DT} introduce a four-term long exact sequence relating the cohomology of this variety and the cohomology of the open set obtained removing the preimage of a general linear section. This sequence is considered to be the mirror of the Clemens-Schmid sequence. It requires neither a restriction to Calabi-Yau varieties  nor a restriction to codimension one fibrations. Recently, \cite{Lee24b} also provided some evidence of Conjecture \ref{conj-higher-dim-base} where the base of a Calabi--Yau fibration is given by $\mathbb P^k$. 

\section{Degenerations for simple normal crossings pairs}
In previous sections, we consider the degeneration of the pair $(X,D)$ where $X$ is a quasi-Fano variety and $D$ is a smooth anticanonical divisor of $X$. The story can be generalized to the case when $D$ is a simple normal crossings divisor. Let 
\[
D_1,\ldots, D_n\subset X
\]
be smooth irreducible divisors and
\[
D=D_1+\ldots+D_n\in |-K_X|
\]
be simple normal crossings. We consider a degeneration of $(X,D)$ into
\[
(X_1,D_{11}+D_{12}+\ldots+D_{1n}), \quad  \text{and } (X_2,D_{21}+D_{22}+\ldots+D_{2n}),
\]
and
\[
D_{1i}\cup_{D_{1i}\cap D_0=D_{2i}\cap D_0}D_{2i}
\]
is a degeneration of $D_i$. We consider the higher rank LG models for $(X,D)$, $(X_1,D_0+D_{11}+D_{12}+\ldots+D_{1n})$ and $(X_2,D_0+D_{21}+D_{22}+\ldots+D_{2n})$. The definition of higher rank LG models when $D$ has more than two components is a direct generalization of rank $2$ LG models defined in Definition \ref{defn:hybrid-LG}. As mentioned in Definition \ref{def-hybrid-LG-inductive}, a rank $n$ LG model can be defined inductively in terms of LG models of lower ranks. We also refer to \cite{Lee24} for the definition.

The topological gluing for higher rank LG models is similar to the topological gluing described in Section \ref{sec:top-gluing}. We glue the rank $(n+1)$ LG models of $(X_1,D_0+D_{11}+D_{12}+\ldots+D_{1n})$ and $(X_2,D_0+D_{21}+D_{22}+\ldots+D_{2n})$ along the first component to get the rank $n$ LG models for $(X,D)$.

A natural question to ask is whether their Hodge numbers and Euler characteristics are still related. To answer this question, we would like to ask what is a reasonable definition of Hodge numbers of higher rank LG models? We examine it for the case when $D$ contains two irreducible components. The following discussions can be generalized to the cases when $D$ contains more than two irreducible components.

Let $(X,D_1+D_2)$ be a log Calabi-Yau simple normal crossings pair. Let $(X^\vee,h=(h_1,h_2))$ be the mirror rank $2$ LG models.

\begin{defn}\label{def-hodge-hybrid}
We define the Hodge numbers of rank $2$ LG models as
\[
h^{p,q}(X^\vee,(h_1,h_2)):=h^{p,q}(X^\vee, h_1^{-1}(t_1)\cup h_2^{-1}(t_2)),
\]
where $t=(t_1,t_2)$ is a regular value of $(h_1,h_2)$.
\end{defn}

\begin{conjecture}
If the rank $2$ LG model $(X^\vee,h)$ is mirror to $(X,D_1+D_2)$, then
\[
h^{d-q,p}(X)=h^{p,q}(X^\vee,(h_1,h_2)).
\]
\end{conjecture}

Here are some reasons for the definition and the conjecture. Let $D_1^\vee=h_1^{-1}(t_1)$ and $D_2^\vee=h^{-1}(t_2)$. Note that $X\setminus (D_1\cup_{D_1\cap D_2} D_2)$ is mirror to $X^\vee$. Therefore, we expect to have the following relation between Euler characteristics under mirror symmetry:
\begin{align}\label{chi-X-D}
\chi(X\setminus (D_1\cup_{D_1\cap D_2} D_2))=(-1)^{\mathrm d}\chi(X^\vee).
\end{align}
Furthermore, by assumption, $(D_1,D_1\cap D_2)$ is mirror to $(D_1^\vee,h_2)$, where $D_1^\vee:=h_1^{-1}(t_1)$. Similarly, $(D_2,D_1\cap D_2)$ is mirror to $(D_2^\vee,h_1)$, where $D_2^\vee:=h_2^{-1}(t_2)$. Then, we have the following relation
\begin{align}\label{chi-D-1-D-2}
\chi(D_1)=(-1)^{\mathrm d-1}\chi(D_1^\vee,h_2^{-1}(t_2)), \quad \text{and } \chi(D_2)=(-1)^{\mathrm d-1}\chi(D_2^\vee,h_1^{-1}(t_1)).
\end{align}
Since $h^{-1}(t):=h_1^{-1}(t_1)\cap h_2^{-1}(t_2)$ is mirror to $D_1\cap D_2$, we have
\begin{align}\label{chi-D-12}
\chi(D_1\cap D_2)=(-1)^{\mathrm d-2}\chi(h^{-1}(t)).
\end{align}
Combining (\ref{chi-X-D}), (\ref{chi-D-1-D-2}) and (\ref{chi-D-12}) together, we have

\begin{align*}
    \chi(X)&=\chi(X\setminus (D_1\cup_{D_1\cap D_2} D_2))+\chi(D_1\cup_{D_1\cap D_2} D_2)\\
    &=(-1)^{\mathrm d}\chi(X^\vee)+\chi(D_1)+\chi(D_2)-\chi(D_1\cap D_2)\\
    &=(-1)^{\mathrm d}\chi(X^\vee)+(-1)^{\mathrm d-1}\chi(D_1^\vee,h_2^{-1}(t_2))+(-1)^{\mathrm d-1}\chi(D_2^\vee,h_1^{-1}(t_1))-(-1)^{\mathrm d-2}\chi(h^{-1}(t))\\
    &=(-1)^{\mathrm d}\chi(X^\vee,h^{-1}(t))+(-1)^{\mathrm d-1}\chi(D_1^\vee,h_2^{-1}(t_2))+(-1)^{\mathrm d-1}\chi(D_2^\vee,h_1^{-1}(t_1))\\
        &=(-1)^{\mathrm d}\left(\chi(X^\vee,h^{-1}(t))-\chi(D_1^\vee,h_2^{-1}(t_2))-\chi(D_2^\vee,h_1^{-1}(t_1))\right)\\
        &=(-1)^{\mathrm d}\left(\chi(X^\vee,D_1^\vee\cup D_2^\vee)\right).
\end{align*}
This is the expected relation between Euler characteristics of $X$ and $(X^\vee,h)$.

Recall that $W:=h_1+h_2:X^\vee\rightarrow \mathbb C$ is a non-proper LG model. Naturally, one can ask what is the relation between Hodge numbers of rank $2$ LG models and Hodge numbers of proper LG model $(\bar{X^\vee},\bar{W})$ which is mirror to $(X,D)$ when $D$ is a smooth anticanonical divisor of $X$? We have the following relation between their Euler characteristics. The relation between their Hodge numbers may be studied elsewhere. 

\begin{theorem}
Given a quasi-Fano variety $X$ with a smooth anticanonical divisor $D\in |-K_X|$. Let $(\bar{X^\vee},\bar{W})$ be the LG mirror of $(X,D)$ and $(X^\vee,h)$ be the rank $2$ LG mirror of $(X,D_1+D_2)$, where $D_1+D_2\in |-K_X|$ is simple normal crossing, then
\[
\chi(\bar{X^\vee},\bar{W}^{-1}(t))=\chi(X^\vee,h_1^{-1}(t_1)\cup h_2^{-1}(t_2)).
\]
\end{theorem}
\begin{proof}
We can consider the degeneration to the normal cone of $X$ with respect to $D_2$. Then $(X,D)$ degenerates to
\[
(X,D_1)\cup_{(D_2,D_1\cap D_2)}(\mathbb P(N_{D_2}\oplus \mathcal O),\tilde D_2).
\]
By Conjecture \ref{conj-gluing-1}, the LG model $(\bar{X^\vee},\bar{W})$ is obtained by gluing two rank $2$ LG models $(X,h)$ and $(P^\vee,h_p)$, where $(P^\vee,h_p)$ is the rank $2$ LG model mirrors to $(\mathbb P(N_{D_2}\oplus \mathcal O),\tilde D_2+D_2)$. By the relative Mayer-Vietoris sequence (\ref{rel-mv-seq}), we have
\[
\chi(\bar{X^\vee},\bar{W}^{-1}(t))=\chi(X^\vee,h_1^{-1}(t_1))+\chi(P^\vee,h_{p,1}^{-1}(t_2)).
\]
Note that 
\[
\chi(X^\vee,h_1^{-1}(t_1))=\chi(X^\vee,h^{-1}(t))-\chi(h_1^{-1}(t_1),h^{-1}(t)).
\]
On the other hand, since $X\cup_{D_2} \mathbb P(N_{D_2}\oplus \mathcal O)$ is smoothable to $X$, by \cite{Lee06}*{Proposition IV. 6}, we have
\[
\chi(X)=\chi(X)+\chi(\mathbb P(N_{D_2}\oplus \mathcal O)-2\chi(D_2).
\]
Hence $\chi(\mathbb P(N_{D_2}\oplus \mathcal O)=2\chi(D_2)$. Recall that $(P^\vee,h_p)$ is the rank $2$ LG model mirrors to $(\mathbb P(N_{D_2}\oplus \mathcal O),\tilde D_2+D_2)$. Therefore, $(P^\vee,h_{p,1})$ is mirror to the complement of the divisor $D_2$ in $\mathbb P(N_{D_2}\oplus \mathcal O)$. Hence,
\begin{align*}
\chi(P^\vee,h_{p,1}^{-1}(t_2))&=(-1)^{\mathrm d}\chi(\mathbb P(N_{D_2}\oplus \mathcal O)\setminus D_2)\\
&=(-1)^{\mathrm d}\left(\chi(\mathbb P(N_{D_2}\oplus \mathcal O))-\chi(D_2)\right)\\
&=(-1)^{\mathrm d}\chi(D_2)\\
&=-\chi(h_2^{-1}(t_2),h^{-1}(t)).
\end{align*}
Finally, this gives
\begin{align*}
&\chi(\bar{X^\vee},\bar{W}^{-1}(t))\\
=&\chi(X^\vee,h_1^{-1}(t_1))+\chi(P^\vee,h_{p,1}^{-1}(t_2))\\
=&\chi(X^\vee,h^{-1}(t))-\chi(h_1^{-1}(t_1),h^{-1}(t))-\chi(h_2^{-1}(t_2),h^{-1}(t))\\
=&\chi(X^\vee,h_1^{-1}(t_1)\cup h_2^{-1}(t_2)).
\end{align*}
\end{proof}
This is exactly what we expected. Therefore, our definition is compatible with the mirror duality between Euler  characteristics when there is a proper LG model.

Now we return to the degeneration of simple normal crossing pairs. We consider the case when $D$ has two irreducible components. The case when $D$ has more than two irreducible components works similarly. We consider the degeneration of $(X,D_1+D_2)$ into $$(X_1,D_{11}+D_{12})\cup_{(D_0,D_0\cap D_{11}+D_0\cap D_{12})} (X_2,D_{21}+D_{22}).$$
The gluing picture of Section \ref{sec:gluing} can be generalized to this case directly. Furthermore, the following relation between Euler characteristics still holds.

\begin{theorem}\label{thm:topo-gluing-2}
Let $X_1$ and $X_2$ be $\mathrm d$-dimensional quasi-Fano varieties which contain the same quasi-Fano hypersurface $D_0$, such that 
\[
K_{X_1}|_{D_0}=-K_{X_2}|_{D_0}.
\]
Let $D_0+D_{11}+D_{12}\in|-K_{X_1}|$ and $D_0+D_{21}+D_{22}\in |-K_{X_2}|$ such that $D_{11}\cap D_0=D_{21}\cap D_0$, $D_{12}\cap D_0=D_{22}\cap D_0$ and
\[
K_{D_{11}}|_{D_0\cap D_{11}}=-K_{D_{21}}|_{D_{0}\cap D_{21}}, \quad K_{D_{12}}|_{D_0\cap D_{12}}=-K_{D_{22}}|_{D_{0}\cap D_{22}},
\]
and
\[
K_{D_{11}\cap D_{12}}|_{D_0\cap D_{11}\cap D_{12}}=-K_{D_{21}\cap D_{22}}|_{D_0\cap D_{21}\cap D_{22}}.
\]
Let $(X_1^\vee,(h_{10},h_{11},h_{12}))$ and $(X_2^\vee,(h_{20},h_{21},h_{22}))$ be the rank $3$ LG models of $(X_{1},D_0+D_{11}+D_{12}\in|-K_{X_1}|)$ and $(X_2,D_0+D_{21}+D_{22}\in |-K_{X_2}|)$. Suppose that the fibres of $(h_{10},h_{11},h_{12})$ and $(h_{20},h_{21},h_{22})$ are topologically the same Calabi-Yau manifold, which is topologically mirror to $D_{012}=D_0\cap D_{11}\cap D_{12}=D_0\cap D_{21}\cap D_{22}$. Let $X$ be a quasi-Fano variety obtained from $X_1\cup_{D_0}X_2$ by smoothing and let $D_1\subset X$ be a quasi-Fano variety obtained from $D_{11}\cup_{D_{012}} D_{21}$ by smoothing and $D_2\subset X$ be a quasi-Fano variety obtained from $D_{12}\cup_{D_{012}} D_{22}$ by smoothing. 
Let $(X^\vee,h=(h_1,h_2))$ be the rank $2$ LG model obtained by gluing rank $3$ LG models $(X_1^\vee,(h_{10},h_{11},h_{12}))$ and $(X_2^\vee,(h_{20},h_{21},h_{22}))$ along the first factor. Then
\[
\chi(X)=(-1)^{\mathrm d}(\chi(X^\vee,h_1^{-1}(t_1)\cup h_2^{-1}(t_2))), 
\]
\[
\chi(D_1)=(-1)^{\mathrm d-1}\chi(h_1^{-1}(t_1),h_2^{-1}(t_2)),
\]
\[
\chi(D_2)=(-1)^{\mathrm d-1}\chi(h_2^{-1}(t_2),h_1^{-1}(t_1)),
\]
and
\[
\chi(D_{1}\cap D_2)=(-1)^{\mathrm d-2}\chi(h^{-1}(t))
\]
for $t=(t_1,t_2)$ a regular value of $h$, where $\chi$ is the Euler number.
\end{theorem}

\begin{proof}
The last three identities follow from Theorem \ref{thm:topo-gluing} and \cite{DHT}*{Theorem 2.3}. We compute the RHS of the first identity. 

By relative Mayer-Vietoris sequence
\begin{align*}
\cdots \rightarrow H^j(X^\vee,(h_1,h_2)^{-1}(t);\mathbb C)\rightarrow H^j(X_1^\vee,(h_{11},h_{12})^{-1}(t_1);\mathbb C)\oplus H^j(X_2^\vee,(h_{21},h_{22})^{-1}(t_2);\mathbb C) \\
\notag \rightarrow H^j(X_1^\vee\cap X_2^\vee, (h_{11},h_{12})^{-1}(t_1)\cap (h_{21},h_{22})^{-1}(t_2);\mathbb C)\rightarrow \cdots,
\end{align*}
similar to the proof of Theorem \ref{thm:topo-gluing}, we have
\begin{align}\label{chi-X-X-1-X-2}
\chi(X^\vee,(h_1,h_2)^{-1}(t))=\chi(X_1^\vee,(h_{11},h_{12})^{-1}(t_1))+\chi(X_2^\vee,(h_{21},h_{22})^{-1}(t_2)).
\end{align}
We can write
\begin{align}\label{chi-X-1}
\chi(X_1^\vee,(h_{11},h_{12})^{-1}(t_1))&=\chi(X_1^\vee,h_{11}^{-1}(t_{11}))+\chi(h_{11}^{-1}(t_{11}),(h_{11},h_{12})^{-1}(t_1))\\
\notag &=(-1)^{\mathrm d}\chi(X_1\setminus (D_0+D_{12}))+(-1)^{\mathrm d-1}\chi(D_{11}\setminus (D_0\cap D_{11}))\\
\notag &=(-1)^{\mathrm d}\left(\chi(X_1)-\chi(D_0)-\chi(D_{12})+\chi(D_0\cap D_{12})-\chi(D_{11})+\chi(D_{0}\cap D_{11})\right).
\end{align}
Similarly,
\begin{align}\label{chi-X-2}
\chi(X_2^\vee,(h_{21},h_{22})^{-1}(t_2))=(-1)^{\mathrm d}\left(\chi(X_2)-\chi(D_0)-\chi(D_{22})+\chi(D_0\cap D_{22})-\chi(D_{21})+\chi(D_{0}\cap D_{21})\right).
\end{align}
Recall that $X_1\cup_{D_0}X_2$ is smoothable to $X$, $D_{11}\cup_{D_0\cap D_{11}=D_0\cap D_{21}}D_{12}$ is smoothable to $D_1$ and $D_{12}\cup _{D_0\cap D_{12}=D_0\cap D_{22}}D_{22}$ is smoothable to $D_2$. Therefore, by \cite{Lee06}*{Proposition IV. 6}, we have
\begin{align}\label{chi-X}
\chi(X)=\chi(X_1)+\chi(X_2)-2\chi(D_0),
\end{align}
\begin{align}\label{chi-D-1}
\chi(D_1)=\chi(D_{11})+\chi(D_{21})-2\chi(D_{0}\cap D_{11}),
\end{align}
and
\begin{align}\label{chi-D-2}
\chi(D_2)=\chi(D_{12})+\chi(D_{22})-2\chi(D_{0}\cap D_{12}).
\end{align}
Hence,
\begin{align*}
    \chi(X^\vee,(h_1,h_2)^{-1}(t))&=\chi(X_1^\vee,(h_{11},h_{12})^{-1}(t_1))+\chi(X_2^\vee,(h_{21},h_{22})^{-1}(t_2))\\
    &=(-1)^{\mathrm d}\left(\chi(X_1)-\chi(D_0)-\chi(D_{12})+\chi(D_0\cap D_{12})-\chi(D_{11})+\chi(D_{0}\cap D_{11})\right)\\
    &\quad +(-1)^{\mathrm d}\left(\chi(X_2)-\chi(D_0)-\chi(D_{22})+\chi(D_0\cap D_{22})-\chi(D_{21})+\chi(D_{0}\cap D_{21})\right)\\
    &=(-1)^{\mathrm d}\left(\chi(X)-\chi(D_1)-\chi(D_2)\right),
\end{align*}
where the first line is (\ref{chi-X-X-1-X-2}); the second line follows from (\ref{chi-X-1}) and (\ref{chi-X-2}); the last line follows from (\ref{chi-X}), (\ref{chi-D-1}) and (\ref{chi-D-2}). This proves the first equality.
\end{proof}

For toric complete intersections, the same computation in Section \ref{sec:toric} also applies to the degeneration of simple normal crossing pairs. We again have the product relation among generalized functional invariants and the Hadamard product relation among relative periods and bases of solutions of the corresponding PDEs. 

\section{Gromov-Witten invariants}\label{sec:GW}

Classical mirror symmetry relates absolute Gromov-Witten invariants with periods (see, for example, \cite{Givental98}, \cite{LLY}). Recently, relative mirror symmetry has been proved in \cite{FTY} to relate relative Gromov-Witten invariants of a smooth pair with relative periods of its mirror. In \cite{DKY}, we explain that the relation between periods gives a relation between absolute and relative Gromov-Witten invariants of a Tyurin degeneration of a Calabi-Yau variety. In this paper, we consider gluing of rank $(n+1)$ LG models to an rank $n$ LG model and gluing relative periods of rank $(n+1)$ LG models to a relative period of a rank $n$ LG model. Higher rank LG models are mirror to simple normal crossings pairs. It is natural to ask if relative periods of higher rank LG models also mirror to certain A-model invariants. In this paper, we consider the type of invariants associated to a simple normal crossing pairs defined in \cite{TY20c} which fits well into our context. In \cite{TY20b}, a mirror theorem has been proved to relate the formal Gromov-Witten invariants of infinite root stacks, which are invariants associated to simple normal crossing pairs and defined in \cite{TY20c}, and periods.   

\begin{remark}
Another well-known invariant for simple normal crossing pairs is (punctured) logarithmic Gromov-Witten invariants of \cite{AC}, \cite{Chen14}, \cite{GS}, \cite{ACGS}. Punctured invariants are essential for the intrinsic mirror symmetry construction in the Gross-Siebert program \cite{GS19}. An approach of using the Gross-Siebert program to prove the Doran-Harder-Thompson conjecture has appeared in \cite{BD}. It would be interesting to relate the period calculation in this paper and in \cite{DKY} to the approach via the Gross-Siebert program. It will be studied elsewhere. 
\end{remark}

We will briefly review the formal Gromov-Witten theory of infinite root stacks and its mirror theorem.

\subsection{Orbifold Gromov-Witten invariants}
Let $\mathcal X$ be a smooth proper Deligne-Mumford stack such that its coarse moduli space $X$ is projective. We consider the moduli space $\bM_{0,l}(\mathcal X, \beta)$ of $l$-pointed genus zero degree $\beta\in H_2(X)$ stable maps to $\mathcal X$.

The genus-zero orbifold Gromov-Witten invariants of $\mathcal X$ are defined as follows
\begin{align}
\left\langle \prod_{i=1}^l \tau_{a_i}(\gamma_i)\right\rangle_{0,l,\beta}^{\mathcal X}:=\int_{[\bM_{0,l}(\mathcal X, \beta)]^{\on{vir}}}\prod_{i=1}^l(\on{ev}^*_i\gamma_i)\bar{\psi}_i^{a_i},
\end{align}
where, 
\begin{itemize}
\item $[\bM_{0,l}(\mathcal X, \beta)]^{\on{vir}}$ is the virtual fundamental class. 
\item 
for $i=1,2,\ldots,l$,
\[
\on{ev}_i: \bM_{0,l}(\mathcal X,\beta) \rightarrow I\mathcal X
\]
is the evaluation map and $I\mathcal X$ is the inertia stack of $\mathcal X$.
\item  $\gamma_i\in H_{\on{CR}}^*(\mathcal X)$ are cohomological classes of the Chen-Ruan orbifold cohomology $H_{\on{CR}}^*(\mathcal X)$.
\item $a_i\in \mathbb Z_{\geq 0}$, for $1\leq i\leq l$.
\item
$\bar{\psi}_i\in H^2(\bM_{0,l}(\mathcal X, \beta),\mathbb Q)$
is the descendant class.
\end{itemize}

In the context of mirror theorems, we need to consider the following generating function of genus zero invariants, called the $J$-function:
\[
J_{\mathcal X}(t,z):=z+t+\sum_{\beta\in \on{NE}(X)}\sum_{l\geq 0}\sum_{\alpha}\frac{q^{\beta}}{l!}\left\langle \frac{\phi_\alpha}{z-\bar{\psi}},t,\ldots,t\right\rangle^{\mathcal X}_{0,l+1,\beta}\phi^{\alpha},
\]
where 
\begin{itemize}
    \item $\on{NE}(X)\subset H_2(X,\mathbb R)$ is the cone generated by effective curves and $\on{NE}(X)_{\mathbb Z}:=\on{NE}(X)\cap (H_2(X,\mathbb Z)/\on{tors})$.
    \item 
$\{\phi_\alpha\}, \{\phi^\alpha\}\subset H^*_{\on{CR}}(\mathcal X)$ are additive bases dual to each other under orbifold Poincar\'e pairing,
\item $t=\sum_{\alpha}t^\alpha\phi_\alpha\in H^*_{\on{CR}}(\mathcal X).$
\end{itemize}
The $J$-function is a slice of Givental's Lagrangian cone. We refer to \cite{Tseng} and \cite{CG} for more details. One can decompose the $J$-function according to the degree of curves
\[
J_{\mathcal X}(t,z)=\sum_{\beta}J_{\mathcal X, \beta}(t,z)q^{\beta}.
\]

\subsection{The formal Gromov-Witten invariants of infinite root stacks}

Let $X$ be a smooth projective variety and let 
\[
D_1,\ldots,D_n\subset X
\]
be smooth irreducible divisors. We assume that
\[
D=D_1+\ldots+D_n
\]
is simple normal crossings. For any index set $I\subseteq\{1,\ldots, n\}$, we define 
\[
D_I:=\cap_{i\in I} D_i.
\]
In particular, we write
\[
D_\emptyset:=X. 
\]
Let 
\[
\vec s=(s_1,\ldots,s_n)\in \mathbb Z^n.
\]
We define
\[
I_{\vec s}:=\{i:s_i\neq 0\}\subseteq \{1,\ldots,n\}.
\]

We consider the multi-root stack
\[
X_{D,\vec r}:=X_{(D_1,r_1),\ldots,(D_n,r_n)},
\]
where $\vec r=(r_1,\ldots,r_n)\in (\mathbb Z_{\geq 0})^n$ and $r_i$'s are pairwise coprime. For the purpose of this paper, we only consider genus zero invariants. By \cite{TY20c}*{Corollary 16}, genus zero orbifold Gromov-Witten invariants of $X_{D,\vec r}$, after multiplying suitable powers of $r_i$, is independent of $r_i$ for $r_i$ sufficiently large. Following \cite{TY20c}*{Definition 18}, the formal Gromov-Witten invariants of $X_{D,\infty}$ is defined as a limit of the corresponding genus zero orbifold Gromov-Witten invariants of $X_{D,\vec r}$.

Let \begin{itemize}
    \item $\gamma_j\in H^*(D_{I_{\vec s^j}})$, for $j\in\{1,2,\ldots,m\}$;
    \item $a_j\in \mathbb Z_{\geq 0}$, for $j\in \{1,2,\ldots,m\}$.
\end{itemize}
The formal genus zero Gromov-Witten invariants of $X_{D,\infty}$ are defined as
\begin{align*}
    \left\langle [\gamma_1]_{\vec s^1}\bar{\psi}^{a_1},\ldots, [\gamma_m]_{\vec s^m}\bar{\psi}^{a_m} \right\rangle_{0,\{\vec s^j\}_{j=1}^m,\beta}^{X_{D,\infty}}:=\left(\prod_{i=1}^n r_i^{s_{i,-}}\right)\left\langle \gamma_1\bar{\psi}^{a_1},\ldots, \gamma_m\bar{\psi}^{a_m} \right\rangle_{0,\{\vec s^j\}_{j=1}^m,\beta}^{X_{D,\vec r}}
\end{align*}
for sufficiently large $\vec r$, where
\begin{itemize}
    \item the vectors
\[
\vec s^j=(s_1^j,\ldots,s_n^j)\in (\mathbb Z)^n, \text{ for } j=1,2,\ldots,m,
\]
satisfy the following condition:
\[
\sum_{j=1}^m s_i^j=\int_\beta[D_i], \text{ for } i\in\{1,\ldots,n\}.
\]
These vectors are used to record contact orders of markings with respect to divisors $D_1,\ldots,D_n$. For orbifold Gromov-Witten theory of multi-root stacks $X_{D,\vec r}$, they record ages of twisted sectors of the inertial stacks $IX_{D,\vec r}$. We refer to \cite{TY20c} for more details.
\item $s_{i,-}$ is the number of markings that have negative contact order with the divisor $D_i$. In other words,
\[
s_{i,-}:=\#\{j: s_i^j<0\}, \text{ for } i=1,2,\ldots, n.
\]

\end{itemize}

\begin{remark}
When the divisor $D$ is smooth, the formal Gromov-Witten invariants of the infinite root stack $X_{D,\infty}$ are simply relative Gromov-Witten invariants of the smooth pairs $(X,D)$ following the work of \cite{ACW}, \cite{TY18} and \cite{FWY}.
\end{remark}

\subsection{Mirror theorems}

A Givental style mirror theorem can be stated as an equality between the $J$-function and the $I$-function via mirror maps. More generally, a mirror theorem can be stated using Givental's formalism (Givental's Lagrangian cone etc.). A mirror theorem for the formal Gromov-Witten invariants of infinite root stacks is proved in \cite{TY20b} as a limit of the mirror theorem for multi-root stacks.

The state space $\mathfrak H$ for the Gromov-Witten theory of $X_{D,\infty}$ is defined as follows:
\[
\mathfrak H:=\bigoplus_{\vec s\in \mathbb Z^n}\mathfrak H_{\vec s},
\]
where 
\[
\mathfrak H_{\vec s}:=H^*(D_{I_{\vec s}}).
\]
Each $\mathfrak H_{\vec s}$ naturally embeds into $\mathfrak H$. For an element $\gamma\in \mathfrak H_{\vec s}$, we write $[\gamma]_{\vec s}$ for its image in $\mathfrak H$. The pairing on $\mathfrak H$ 
\[
(-,-):\mathfrak H \times \mathfrak H\rightarrow \mathbb C
\]
is defined as follows: for $[\alpha]_{\vec s}$ and $[\beta]_{\vec s^\prime}$, define
\begin{equation}\label{eqn:pairing}
\begin{split}
([\alpha]_{\vec s},[\beta]_{\vec s^\prime}) = 
\begin{cases}
\int_{D_{I_{\vec s}}} \alpha\cup\beta, &\text{if } \vec s=-\vec s^\prime;\\
0, &\text{otherwise. }
\end{cases}
\end{split}
\end{equation}
The pairing on the rest of the classes is generated by linearity.

When $D_i$ are nef, the non-extended $I$-function is
\begin{align}\label{I-snc}
I_{X_{D,\infty}}(q,t,z):=\sum_{\beta\in \on{NE}(X)} J_{X, \beta}(t,z)q^{\beta}
\prod_{i=1}^n\prod_{0< a< d_i}(D_i+az)[\textbf{1}]_{ (-d_1,-d_2,\ldots, -d_n)}.
\end{align}

A mirror theorem for infinite root stacks is proved in \cite{TY20b}*{Section 4} (see also \cite{TY20c}*{Theorem 29}) which states that the $I$-function $I_{X_{D,\infty}}$ lies in Givental's Lagrangian cone for $X_{D,\infty}$. When $-K_X-D$ is nef, the $I$-function equals to the $J$-function via a change of variable called the mirror map. When $n=1$, this specializes to the mirror theorem for relative Gromov-Witten invariants of smooth pairs $(X,D)$ proved in \cite{FTY}*{Theorem 1.4}. One can also write down the extended $I$-function which is more complicated. Again, by \cite{TY20c}*{Theorem 29} the extended $I$-function lies in Givental's Lagrangian cone. For our propose, we only write down the part of the extended $I$-function of $X_{D,\infty}$,denoted by $I_{X_{D,\infty},0}(q,x,t,z)$, that takes value in $\mathfrak H_{(0,\ldots, 0)}:=H^*(X)$:
\begin{align*}
I_{X_{D,\infty},0}(q,x,t,z):=\sum_{\substack{\beta\in \on{NE}(X),(k_{i1},\ldots,k_{im})\in (\mathbb Z_{\geq 0})^m\\ \sum_{j=1}^m jk_{ij}= d_i, 1\leq i \leq n}  } & J_{X, \beta}(t,z)q^{\beta}\frac{\prod_{i=1}^n\prod_{j=1}^m x_{ij}^{k_{ij}}}{z^{\sum_{i=1}^n\sum_{j=1}^m k_{ij}}\prod_{i=1}^n\prod_{j=1}^m(k_{ij}!)}\\
& \cdot\left(\prod_{i=1}^n\prod_{0<a\leq d_i}(D_i+az)\right),
\end{align*}
where variables $x_{ij}$'s are used to record tangency conditions: contact order $j$ along the divisor $D_i$. 

We can specialize it to the toric complete intersections that we considered in Section \ref{sec:toric}. Let $X$ be a complete intersection in a toric variety $Y$ defined by a generic section of $E=L_0\oplus L_1\oplus\cdots \oplus  L_s$, where each $L_l$ is a nef line bundle. Let $\rho_l=c_1(L_l)$ and $\mathbf D_j$, $1\leq j \leq m$, be toric divisors. Let $D:=D_1+\ldots+D_n\in |-K_X|$ and assume that $D$ is simple normal crossings and $D_i$ are nef for $1\leq i \leq n$. Then the $I$-function for $(X,D)$ can be written as
\begin{align}\label{I-func-X-D}
I_{X_{D,\infty},0}(q,x,t,z):=e^{\sum_{i=0}^r p_i\log q_i}\sum_{d\in \on{NE}(Y)_{\mathbb Z}}&\prod_{j=1}^m \left(\frac{\prod_{k=-\infty}^0(\mathbf D_j+kz)}{\prod_{k=-\infty}^{\langle \mathbf D_j,d\rangle}(\mathbf D_j+kz)}\right)\left(\prod_{l=0}^s \prod_{k=1}^{\langle \rho_l, d\rangle }(\rho_l+kz)\right)q^d\\
\notag & \frac{\prod_{i=1}^n\prod_{j=1}^m x_{ij}^{k_{ij}}}{z^{\sum_{i=1}^n\sum_{j=1}^m k_{ij}}\prod_{i=1}^n\prod_{j=1}^m(k_{ij}!)}
 \cdot\left(\prod_{i=1}^n\prod_{0<a\leq d_i}(D_i+az)\right).
\end{align}
Note that the nefness assumption on $D_i$ can be removed if $D_i$ is coming from a toric divisor of $Y$.

Similar to the discussion in \cite{DKY}*{Section 6.4}, the information of periods in Section \ref{sec:toric-period} can be extracted from the $I$-function (\ref{I-func-X-D}). By relative mirror theorem \cite{TY20b}*{Theorem 29}, these periods compute genus zero Gromov-Witten invariants of $(X,D)$.

Recall that, we consider the degeneration of $(X,D)$ into $(X_1,D_1)\cup_{(D_0,D_{12})}(\tilde X_2,\tilde D_2)$, where $X_1\cap \tilde X_{2}=D_0$ and $D_1\cap \tilde D_2=D_{12}$. Then $I^X,I^{X_1},I^{\tilde X_2}$ and $I^{D_0}$ in Section \ref{sec:toric-period} are extracted from the $I$-functions of $(X,D), (X_1,D_0+D_1), (\tilde X_2,D_0+\tilde D_2)$ and $(D_0,D_{12})$ respectively. 

As mentioned in Remark \ref{rmk-periods}, the gluing formula still holds if we remove the common factor involving $\rho_{s+1}$ from the $I$-functions. In here, it means removing the factors involving $D_i$, for $1\leq i\leq n$. Then we obtain a relation among $I$-functions of $X, (X_1,D_1), (\tilde X_2,\tilde D_2)$ and $D_0$. 

Through mirror symmetry, the gluing formulae among periods provides a relation among Gromov-Witten invariants. On the A-model side, there are degeneration formulae relating Gromov-Witten invariants. In future work, we will study how these formulae from A-model and B-model are related. We expect the following.

\begin{conjecture}
The gluing formula for period integrals in the form of Theorem \ref{thm:general-case-I-function} holds for Tyurin degenerations of Calabi--Yau varieties (or log Calabi--Yau varieties). Furthermore, the degeneration formula for Gromov--Witten invariants implies the gluing formula for period integrals.
\end{conjecture}

\bibliographystyle{amsxport}
\bibliography{universalreferences.bib}

\end{document}